\documentclass{article}

\usepackage{amsthm, amsmath, amssymb}

\usepackage{amsmath}
\usepackage{amssymb}
\usepackage{amsthm}
\usepackage{setspace}
\usepackage[utf8]{inputenc}
\usepackage{enumerate}
\usepackage{amsmath,xspace,amssymb,mathrsfs}
\usepackage{cjhebrew}

\input xy
\xyoption{all}
\xyoption{2cell}

\UseAllTwocells
\CompileMatrices

\usepackage{}
\usepackage{PageSetup}
\usepackage{Alphabets}
\usepackage{Definitions}
\usepackage{Environments}
\usepackage{TikzSetup}
\usepackage{AuthorInfo}

\theoremstyle{plain} 
\clevertheorem{theorem}{Theorem}{Theorems}
\newtheorem*{theorem*}{Theorem}
\clevertheorem{proposition}{Proposition}{Propositions}
\clevertheorem{lemma}{Lemma}{Lemmatta}
\clevertheorem{corollary}{Corollary}{Corollaries}
\clevertheorem{conjecture}{Conjecture}{Conjectures}
\newtheorem*{conjecture*}{Conjecture}

\theoremstyle{definition} 
\clevertheorem{definition}{Definition}{Definitions}
\newtheorem*{definition*}{Definition}
\clevertheorem{example}{Example}{Examples}
\newtheorem*{example*}{Example}

\clevertheorem{construction}{Construction}{Constructions}
\newtheorem*{construction*}{Construction}

\clevertheorem{visualization}{Visualization}{Visualizations}
\newtheorem*{visualization*}{Visualization}

\newtheorem*{question*}{Question}

\newtheorem*{claim*}{Claim}

\usetikzlibrary{
  decorations.pathmorphing,
  shadows.blur,
  backgrounds,
  calc
}

\title{An Inductive Strategy Towards a Solution to the Generalized Homotopy Hypothesis}
\author{Johnathon Taylor}
\date{}

\begin{document}

\maketitle

\begin{abstract}
Using the theory of distributive series of monads, we construct an $(\infty,0)$-coherator called the \emph{inductive coherator}. The category of models out of the inductive coherator serve as a model for $\infty$-groupoids that possess an underlying globular set. Once we establish the construction for the inductive coherator, we provide the framework for an inductive strategy to prove the Generalized Homotopy Hypothesis obtained by transferring model structure off of the category of $n$-groupoids onto the category of $(n+1)$-groupoids. Moreover, we provide a necessary and sufficient condition for the transfer of model structure to be successful. We conclude by showing if the transfer of model structure may be completed successively, then the Generalized Homotopy Hypothesis is true.
\end{abstract}

\section{Introduction}

In \emph{Pursuing Stacks}, Alexander Grothendieck hypothesized that models out of certain limit sketches, now called \emph{$(\infty,0)$-coherators}, should model all homotopy types in a statement called the \emph{Homotopy Hypothesis}. Moreover, Grothendieck hypothesized that $n$-truncated models should model all homotopy $n$-types, in what is called the \emph{Generalized Homotopy Hypothesis}.  Maltsiniotis rigorously defined $(\infty,0)$-coherators in \cite{Geo2} and the models are called Grothendieck $\infty$-groupoids. Henry  and Lanari  have put in a substantial effort to provide the structure of a semi-model structure on $n$-groupoids for all $n$ (see \cite{Hen2, HenLan,lanari2018semimodelstructuregrothendieckweak, lane2}). The semi-model structure is known to exist if the pushout conjecture (see Conjecture 5.3.3 of \cite{Hen2}) is true or we can build a path object for every cofibrant infinity groupoid (see \cite{lane2}). When such a semi-model structure exists, the Homotopy Hypothesis is known to be true.

In this paper, we examine a hypothetical solution to the Homotopy Hypothesis that relies on induction. We build a framework that allows us to hypothetically transfer a model structure from the category of $n$-groupoids onto the category of $(n+1)$-groupoids for all finite $n$. The category of $\infty$-groupoids is a limit of the functors that truncate an $(n+1)$-groupoid to an $n$-groupoid. We use this to equip the category of $\infty$-groupoids with a hypothetical limit model structure. Moreover, we show that the existence of the canonical model structure on $n$-groupoids for all finite $n$ is equivalent to the existence of the canonical model structure on $\infty$-groupoids. A direct corollary of this is the following: we can boil down proving the Homotopy Hypothesis to showing that can certain weak equivalences are preserved under a free functor (see Definition \ref{The Free Pushout Condition} and Prosotion \ref{induc_transfer}). 

As an aside, we have a reason to believe the semi-model structure is a full fledge model structure. We will elaborate on how we have reached this conclusion in a future paper. For a short synposis, we will construct \emph{weak orientals} for Grothendieck $\infty$-groupoids. The weak orientals will induce a realization-nerve adjunction between simplicial sets and Grothendieck $\infty$-groupoids. We will use this adjunction to put a hypothetical right transfer model structure on Grothendieck $\infty$-groupoids, prove that the hypothetical right transfer model structure coincides with the hypothetical canonical model structure and lay out a detailed plan to prove the Homotopy Hypothesis. Moreover, we use the existence of the weak orientals to build the \emph{singular homology} for Grothendieck $\infty$-groupoids. The prerequisite material to discuss the construction of the weak orientals is extensive and deserves its own paper. The results of this paper are reflective of this believe discussed; moreover, all the results obtained in this paper may be rephrased in terms of semi-model structures. Our paper serves as a small piece for a larger program to enlarge the theory and introduce algebraic variants for Grothendieck $\infty$-groupoids.

\section{Outline}
We organize this paper as follows. In Section 3, we provide background on distributive laws and series. We use distributive series of monads to inductively construct monads that output higher dimensional objects. In Section 4 and Section 5, we provide background on Grothendieck $\infty$-groupoids and build a distributive series of monads that induces an $(\infty,0)$-coherator, called the \emph{inductive coherator}. The tower
\[
IC_0\to IC_1\to\cdots\to IC_n\to IC_{n+1}\to\cdots
\]
whose colimit is the inductive coherator encodes the classical procedure of weak enrichment in higher category theory. 

In Section 6, we take the tower for the inductive coherator and begin the process of defining a chain of limit sketches that points in the opposite direction, where $IC_0^{\leq -1}$ is a one object category.
\[
\cdots \to IC_{n+2}^{\leq n+1}\to IC_{n+1}^{\leq n}\to\cdots\to IC_{1}^{\leq 0}\to IC_0^{\leq -1}
\]
This will come equipped with a map $\tr_n:IC_{n+1}\to IC_{n+1}^{\leq n}$ of limit sketches for all $n\geq 0$ such that the following square commutes for all $n\geq 0$. 

\[
\begin{tikzpicture}[node distance=2cm,auto]
	\node (A) {$IC_{n+1}$};
	\node (B)[right of=A] {$IC_{n+2}$};
	\node(C)[below of=A] {$IC_{n+1}^{\leq n}$};
	\node (D)[right of=C] {$IC_{n+2}^{\leq n+1}$};
	\draw[->] (A) to node {}(B);
	\draw[->,swap] (A) to node{$\tr_n$}(C);
	\draw[->] (B) to node{$\tr_{n+1}$}(D);
	\draw[->,swap] (D) to node {$\tr$}(C);
\end{tikzpicture}
\]
Moreover, we define the category of $n$-groupoids to be the category of models out of $IC_{n+1}^{\leq n}$ and this square induces a square of adjoints.
\[
\begin{tikzpicture}[>=stealth, thick, node distance=4cm]

\node (A) {$\Mod(IC_{n+2})$};
\node (B) [right of=A] {$\Mod(IC_{n+1})$};
\node (C) [below of=A] {$\mathpzc{(n+1)Gpd}$};
\node (D) [right of=C] {$\ngpd$};

\draw[transform canvas={yshift=0.3ex},->] (A) to node[above] {$U$} (B);
\draw[transform canvas={yshift=-0.3ex},<-] (A) to node[below] {$\free$} (B);

\draw[transform canvas={yshift=-0.3ex},->] (D) to node[below] {$D$} (C);
\draw[transform canvas={yshift=0.3ex},<-] (D) to node[above] {$\tr$} (C);

\draw[transform canvas={xshift=0.3ex},->] (A) to node[left] {$D$} (C);
\draw[transform canvas={xshift=-0.3ex},<-] (A) to node[right] {$\tr$} (C);

\draw[transform canvas={xshift=-0.3ex},->] (B) to node[left] {$\tr$} (D);
\draw[transform canvas={xshift=0.3ex},<-] (B) to node[right] {$D$} (D);

\end{tikzpicture}
\]

In Sections 7 to 10, we define a notion for the \emph{canonical model structures}, if they exist, on $\Mod(IC_{n+1})$ and $n$-groupoids for all $n\geq 0$. We show that the canonical model structure on $\Mod(IC_{n+1})$ exists if and only if the canonical model structure on $n$-groupoids for all $n\geq 0$. Moreover, we provide a necessary and sufficient condition to prove that the existence of the canonical model structure on $\Mod(IC_{n+1})$ implies the existence of the canonical model structure on $\Mod(IC_{n+2})$ for any $n\geq 0$. We finish in Section 11 by defining the \emph{canonical model structure} on $\infty$-groupoids, if it exists, as a type of limit model structure. We prove that the canonical model structure for $\infty$-groupoids exists if and only if the canonical model structure on $n$-groupoids exists for all finite $n$. 




\section{Distributive Laws and Series}
We begin by providing background on distribitutive series of monads. Eugenia Cheng first coined distributive series of monads in \cite{EC1} where she constructed a distributive series of monads whose composite monad is the free strict $n$-category for all finite $n$. We extends the definition of distributive series of monads to include a $\aleph_0$-indexed collection of monads. We will provide a name for when the induced transfinite composite of monads has the structure of a monad. 
\begin{definition}\label{mon_over_mon}
	Let $C$ be a category and $(S,\eta^S,\mu^S)$ and $(T,\eta^T,\mu^T)$ be monads on $C$.  A {\em distributive law of $S$ over $T$} consists of a natural transformation $\lambda : ST \Rightarrow TS$ such that the following diagrams commute.
	
	
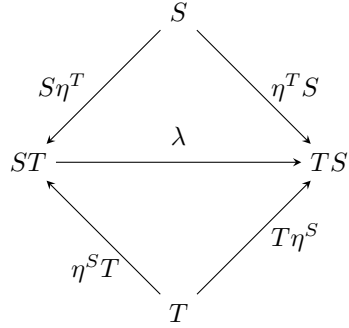
\[
		\begin{tikzpicture}[node distance=2cm]
			\node (A){$ST$};
			\node (B)[right of=A]{};
			\node(C)[right of=B]{$TS$};
			\node(D)[below of=B] {$T$};
			\node (E)[above of=B]{$S$};
			\draw[->] (A) to node[black,above=3] {$\lambda$} (C);
			\draw[->,swap](D) to node [black,left=3,below=3] {$\eta^ST$} (A);
			\draw[->] (D) to node [black,right=3] {$T\eta^S$} (C);
			\draw[->,swap](E) to node [black,left=3] {$S\eta^T$} (A);
			\draw[->] (E) to node [black,right=3] {$\eta^TS
				$} (C);
		\end{tikzpicture}
	\]
	
	
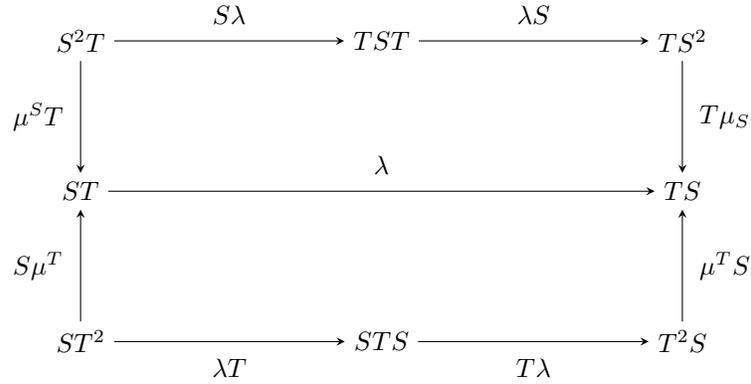
\[
		\begin{tikzpicture}[node distance=2cm]
			\node (A) {$S^2T$};
			\node (B)[below of =A]{$ST$};
			\node (C)[below of=B]{$ST^2$};
			\node (D)[right of=C,node distance=4cm]{$STS$};
			\node (E) [above of=D]{};
			\node (F) [above of=E]{$TST$};
			\node (G) [right of=F,node distance=4cm]{$TS^2$};
			\node (H) [below of=G]{$TS$};
			\node (I)[below of=H]{$T^2S$};
			\draw[->,swap] (A) to node[black,left=3] {$\mu^ST$}(B);
			\draw[->] (A) to node [black,above=3] {$S\lambda$}(F);
			\draw[->] (F) to node[black,above=3] {$\lambda S$}(G);
			\draw[->](G) to node[black,right=3] {$T\mu_S$}(H);
			\draw[->](B) to node[black,above=3] {$\lambda$}(H);
			\draw[->] (C) to node[black,left=3] {$S\mu^T$}(B);
			\draw[->] (I) to node[black, right=3]{$\mu^TS$} (H);
			\draw[->,swap](C) to node[black, below=3]{$\lambda T$} (D);
			\draw[->,swap] (D) to node[black,below=3]{$T\lambda$}(I);
		\end{tikzpicture}
\]
\end{definition}

These diagrams tell us that "$\lambda$ interacts coherently with the monad structure for $S$ and $T$". The following theorem of Beck tells us that distributive laws allows us to compose two monads to form a new monad. 

\begin{proposition}(Beck, Proposition 1 of \cite{Beck_Dist})\label{comp_monad}
Let $(S,\eta^S,\mu^S)$, $(T,\eta^T,\mu^T)$ be monads on $C$ and $\lambda:ST\Rightarrow TS$ be a distributive law. Then the triple 
\[
(ST,S\eta^T\circ \eta^S,S\mu^T\circ\mu^ST^2\circ S\lambda T)
\]
is a monad.
\end{proposition}

Moreover, Proposition 9.2.3, Proposition 9.2.5, and Proposition 9.2.6 of \cite{BarrWells} tells us about the new monad that arises from a distributive law and its category of algebras.

\begin{example}
Let $C=Set$, $T$ be the free abelian group monad, and $S$ be the free commutative monoid monad. There is a distributive law
\[
\lambda:TS \Rightarrow ST 
\]
that encodes the distributivity law axiom for rings.
\[
(a+b)\times(c+d) \longmapsto a\times c+a\times d+b\times c+b\times d.
\]
Then the induced monad
\[
(ST,S\eta^T\circ \eta^S,S\mu^T\circ\mu^ST^2\circ S\lambda T)
\]
is the monad for rings.
\end{example}

\begin{example}
Let $C$ be the category of $2$-globular sets, i.e the diagrams in $\Set$
\[
		\begin{tikzpicture}[node distance=2cm]
		\node (A) {$X_2$};
			\node (B)[right of=A]{$X_1$};
			\node (C)[right of=B]{$X_0$};
	\draw[transform canvas={yshift=0.3ex},->] (A) to node[above=3] {$s$} (B);
			\draw[transform canvas={yshift=-0.3ex},->,swap](A) to node[below=3] {$t$} (B);
			\draw[transform canvas={yshift=0.3ex},->] (B) to node[above=3] {$s$} (C);
			\draw[transform canvas={yshift=-0.3ex},->,swap](B) to node[below=3] {$t$} (C);

	\end{tikzpicture}
	\]	
	subject to the relations 
	\[
	s\circ s=s\circ t
	\]
	and 
	\[
	t\circ s=t\circ t.
	\]
Let $S$ be the monad for vertical composition of $2$-cells and $T$ be the monad for horizontal composition of $2$-cells and $1$-cells. There is a distributive law $\lambda:ST\Rightarrow TS$ that encodes the interchange axiom. Then the induced monad
\[
(ST,S\eta^T\circ \eta^S,S\mu^T\circ\mu^ST^2\circ S\lambda T)
\]
is the monad for strict $2$-categories.
\end{example}

\begin{definition}\label{dist_ser_monads_extended}
	A \emph{distributive series of monads} consists of monads $T_0,T_1,\cdots,T_n,\cdots$ together with a distributive law $\lambda^{i,j}:T_iT_j\Rightarrow T_jT_i$
	for $j>i\geq 0$ such that the diagram 
	\[
\begin{tikzpicture}[node distance=1.5cm]
	\node (A){$T_iT_jT_k$};
	\node (B)[right of=A]{};
	\node (C)[above of =B]{$T_jT_iT_k$};
	\node (D)[below of =B]{$T_iT_kT_j$};
	\node (E)[right of=C,node distance=3cm]{$T_jT_kT_i$};
	\node (F)[right of=D,node distance=3cm]{$T_kT_iT_j$};
	\node (G)[right of=B]{};
	\node (H)[right of =G,node distance=3cm]{$T_kT_jT_i$};
	\draw[->](A) to node[left=3]{$\lambda^{i,j}T_k$}(C);
	\draw[->,swap](A) to node[left=3]{$T_i\lambda^{j,k}$}(D);
	\draw[->](C)to node[above=3]{$T_j\lambda^{i,k}$}(E);
	\draw[->,swap](D) to node[below=3]{$\lambda^{i,k}T_j$}(F);
	\draw[->](E) to node[right=3]{$\lambda^{j,k}T_i$}(H);
	\draw[->,swap](F) to node[right=3]{$T_k\lambda^{i,j}$}(H);
\end{tikzpicture}
\]
commutes for $k>j>i\geq 0$.  We write 
	\[
	T:=((T_i)_{i=0}^\infty,(\lambda_{i,j})_{i>j})
	\]
	to denote this distributive series of monads. 
\end{definition}

This is an extension of Definition 2.3 of \cite{EC1}.

    \begin{definition}
A map $f:T\to S$ of distributive series consists of a monad morphism 
\[
f_i:T_i\to S_i
\]
for all $i\geq 0$ such that the square 
\[
\begin{tikzpicture}[node distance=2cm,auto]
	\node (A) {$T_iT_j$};
	\node (B)[right of=A] {$T_jT_i$};
	\node(C)[below of=A] {$S_iS_j$};
	\node (D)[right of=C] {$S_jS_i$};
	\draw[->] (A) to node {$\lambda_{i{,}j}$}(B);
	\draw[->,swap] (A) to node{$f_if_j$}(C);
	\draw[->] (B) to node{$f_jf_i$}(D);
	\draw[->,swap] (C) to node {$\lambda'_{i,j}$}(D);
\end{tikzpicture}
\]
commutes for $i>j\geq 0$.
\end{definition}

\begin{notation}
We use $\DistSys$ to denote the category of distributive systems and maps between them.
\end{notation}

\begin{theorem}\label{thmChengReIm_corr}
	Let $((T_i)_{i=0}^\infty,(\lambda_{i,j})_{i>j})$ be a distributive series of monads. Then for all $n\geq 2$ and $1\leq i < n$  we have induced monads 
	\[T_i T_{i-1} \cdots T_0 \quad \mbox{and} \quad T_{n} T_{n-1} \cdots T_{i+1}\]
	
	\noindent together with a distributive law of \ $T_i T_{i-1} \cdots T_0$\  over \ $T_{n} T_{n-1} \cdots T_{i+1}$\ i.e. 
	\[(T_i T_{i-1} \cdots T_0) (T_{n} T_{n-1} \cdots T_{i+1}) \Rightarrow (T_{n} T_{n-1} \cdots T_{i+1})(T_i T_{i-1} \cdots T_0)\]
	
	\noindent given by the composites of the $\lambda^{ij}$.  Moreover, all the monad structures on \ $T_n T_{n-1} \cdots T_0$ induced by combining the structure maps of the monads and distributive laws are the same for $n\geq 2$.
\end{theorem}

\begin{proof}
The proof of this is identical to the one Cheng provided for Theorem 2.1 of \cite{EC1}. Hence, this is really a corollary of Theorem 2.1 of \cite{EC1}.
\end{proof}

Given a distributive series of monads, we would like for the transfinite compostion of monads to induce a monad structure. Therefore we give notation for the finite composites to start.

\begin{definition}
	Let $n\geq 0$, $C$ be a category and 
	\[
	T:=((T_i)_{i=0}^\infty,(\lambda_{i,j})_{i>j})
	\]
    be a distributive series of monads. The \emph{$n$th associated monad} is the monad structure 
on $T_nT_{n-1}\cdots T_0$ induced in Theorem \ref{thmChengReIm_corr} and is denoted by $(\hat{T}_n,\hat{\eta}_n,\hat{\mu}_n)$.
\end{definition}

The first difficulty we must address to obtain a transfinite composition of monads is whether the underlying pointed endofunctor exists.

\begin{definition}\label{ind_funct_on_dist}
	Let $C$ be a category and 
	\[
	T:=((T_i)_{i=0}^\infty,(\lambda_{i,j})_{i>j})
	\]
	be a distributive series of monads. The \emph{associated pointed endofunctor of the distributive series}  $\hat{T}$, if it exists, is defined to be
	to be the colimit of the diagram
	\[
	1_{C}\xrightarrow{\eta^0} T_0\xrightarrow{\eta^1_{T_0}} T_1T_0\xrightarrow{\eta^2_{T_1T_0}}\cdots
	\] \label{diag_for_ass_pntd_endo}
	in the category of endofunctors on $C$ with pointing given to be the map
	\[
	\hat{\eta}:1_C\Rightarrow \hat{T}
	\]
	induced by taking colimits. 
\end{definition}
Even if the underlying pointed endofunctor $(\hat{T},\hat{\eta})$ exists, we will not always be able to provide a monad structure. We give a special name to the distributive series when it can. 
	\begin{definition}\label{def_of_compl_distr_ser_of_monads}
	Let $C$ be a category and 
    \[
	T:=((T_i)_{i=0}^\infty,(\lambda_{i,j})_{i>j})
	\]
	be a distributive series of monads on $C$ whose associated pointed endofunctor $(\hat{T},\hat{\eta})$ exists. We say $T$ is \emph{completable} if there is a natural transformation $\hat{\mu}:\hat{T}^2\Rightarrow\hat{T}$ such that $(\hat{T},\hat{\eta},\hat{\mu})$ is a monad.
\end{definition}

\begin{example}
We may extend Theorem 4.10 of \cite{EC1} to form a completable distributive series of monads
	\[
	T:=((T_i)_{i=0}^\infty,(\lambda_{i,j})_{i>j})
	\]
    on $\bG-\Set$ whose induced monad $(\hat{T},\hat{\eta},\hat{\mu})$ on $\bG-\Set$ is the free strict $\omega$-category monad $(-)^*$ of \cite{Lein}.
\end{example}

\section{Grothendieck Infinity Groupoids}
\noindent All of the material of this section is taken from \cite{Dim1}, \cite{Bo2}, and \cite{Malt}. We present it here to give a self-contained treatment.

\begin{notation}
	The \emph{globe category} $\bG$ is the category
\[
		\begin{tikzpicture}[node distance=2cm]
		\node (A) {$0$};
			\node (B)[right of=A]{$1$};
			\node (C)[right of=B]{$2$};
			\node (D)[right of=C]{$\cdots$};
	\draw[transform canvas={yshift=0.3ex},->] (A) to node[above=3] {$s$} (B);
			\draw[transform canvas={yshift=-0.3ex},->,swap](A) to node[below=3] {$t$} (B);
			\draw[transform canvas={yshift=0.3ex},->] (B) to node[above=3] {$s$} (C);
			\draw[transform canvas={yshift=-0.3ex},->,swap](B) to node[below=3] {$t$} (C);
			\draw[transform canvas={yshift=0.3ex},->] (C) to node[above=3] {$s$} (D);
			\draw[transform canvas={yshift=-0.3ex},->,swap](C) to node[below=3] {$t$} (D);
	\end{tikzpicture}
	\]	
	subject to the relations 
	\[
	s\circ s=t\circ s
	\]
	and 
	\[
	t\circ t=s\circ t.
	\]
\end{notation}

\begin{definition}
	A \emph{globular object} $X$ in a category $C$ is a  presheaf over $\bG$. 
    \[
    X:\bG^\op\to C
    \]
\end{definition}

\begin{example}\label{example_important}
Define a functor $D :\bG \to \Top$ by letting $D$ send $i$ to the $i$-dimensional ball
\[ D^i = \{x \in \mathbb{R}^i : ||x|| \le 1\} \] for $i\geq 0$.
For $i \ge 1$, the morphisms $s$ and $t$ are sent by $D$
to $\sigma^i$ and $\tau^i$, respectively, defined by
\[ \sigma^i(x) = (x, \sqrt{1 - \Vert x\Vert^2}) \quad\text{and}\quad
\tau^i(x) = (x, -\sqrt{1 - \Vert x\Vert^2}) 
\]
for $x \in D^{i-1}$. These morphisms are the inclusions of the two hemispheres of $D^{i-1}$ into $D^i$. This induces a globular object 
\[
\Top(D^{(-)},X):\bG^\op\to \Top
\]
 of $\Top$ for every space $X$.
\end{example}

\begin{definition}
A \emph{zig-zag diagram} in $\bG$ is a diagram of the form. 

\begin{center}
	\begin{tikzpicture}[node distance=1cm, auto]
		
		\node (A) {$n_1$};
		\node (B) [right of=A]{} ;
		\node (C) [above of=B] {$n_2$};
		\node (D) [below of=C]{};
		\node(E)[right of=D]{$n_3$};
		\node (F) [right of=E]{} ;
		\node (G) [above of=F] {$n_4$};
		\node (H) [right of=G]{} ;
		\node (I) [below of=H] {$n_5$};
		\node (J)[right of=I]{};
		\node (X) [node distance=0.5cm, above of=J] {$\cdots$};
		
		\node (K)[right of=J] {$n_{2k-1}$};
		\node (L) [right of=K]{} ;
		\node (M) [above of=L] {$n_{2k}$};
		\node (N) [below of=M]{};
		\node(O)[right of=N]{$n_{2k+1}$};
		
		\draw[->,swap] (C) to node {$s$} (A);
		\draw[->] (C) to node {$t$} (E);
		\draw[->,swap] (G) to node {$s$} (E);
		\draw[->] (G) to node {$t$} (I);
		
		\draw[->,swap] (M) to node {$s$} (K);
		\draw[->] (M) to node {$t$} (O);
		
	\end{tikzpicture}
\end{center}
in $\bG$.
\end{definition}

\begin{definition}
A \emph{table of dimensions} in $\bG$ consists of a $(2k+1)$-tuple of natural numbers $\vec{n}=(n_1,\dots,n_{2k+1})$ such that 
\begin{center}
	\begin{tikzpicture}[node distance=1cm, auto]
		
		\node (A) {$n_1$};
		\node (B) [right of=A]{} ;
		\node (C) [above of=B] {$n_2$};
		\node (D) [below of=C]{};
		\node(E)[right of=D]{$n_3$};
		\node (F) [right of=E]{} ;
		\node (G) [above of=F] {$n_4$};
		\node (H) [right of=G]{} ;
		\node (I) [below of=H] {$n_5$};
		\node (J)[right of=I]{};
		\node (X) [node distance=0.5cm, above of=J] {$\cdots$};
		
		\node (K)[right of=J] {$n_{2k-1}$};
		\node (L) [right of=K]{} ;
		\node (M) [above of=L] {$n_{2k}$};
		\node (N) [below of=M]{};
		\node(O)[right of=N]{$n_{2k+1}$};
		
		\draw[->,swap] (C) to node {$s$} (A);
		\draw[->] (C) to node {$t$} (E);
		\draw[->,swap] (G) to node {$s$} (E);
		\draw[->] (G) to node {$t$} (I);
		
		\draw[->,swap] (M) to node {$s$} (K);
		\draw[->] (M) to node {$t$} (O);
		
	\end{tikzpicture}
\end{center}
forms a zig-zag diagram in $\bG$.
\end{definition}

\begin{definition}
Let $\vec{n}=(n_1,\dots,n_k)$ be a table of dimensions. The \emph{height} is
	\[
	\hgt(\vec{n})=\max\{n_1,\dots,n_k\}.
	\]
\end{definition}

\begin{definition}
Let $A:\bG^\op\to C$ be a globular object of $C$ and $\vec{n}=(n_1,\dots,n_{2k+1})$ be a table of dimensions. Then the limit of
\begin{center}
		\begin{tikzpicture}[node distance=1.25cm, auto]
			
			\node (A) {$An_1$};
			\node (B) [right of=A]{} ;
			\node (C) [above of=B] {$An_2$};
			\node (D) [below of=C]{};
			\node(E)[right of=D]{$An_3$};
			\node (F) [right of=E]{} ;
			\node (G) [above of=F] {$An_4$};
			\node (H) [right of=G]{} ;
			\node (I) [below of=H] {$An_5$};
			\node (J)[right of=I]{};
			\node (X) [node distance=0.75cm, above of=J] {$\cdots$};
			
			\node (K)[right of=J] {$An_{k-2}$};
			\node (L) [right of=K]{} ;
			\node (M) [above of=L] {$An_{k-1}$};
			\node (N) [below of=M]{};
			\node(O)[right of=N]{$An_k$};
			
			\draw[->] (A) to node {$As$} (C);
			\draw[->,swap] (E) to node {$At$} (C);
			\draw[->] (E) to node {$As$} (G);
			\draw[->,swap] (I) to node {$At$} (G);
			
			\draw[->] (K) to node {$As$} (M);
			\draw[->,swap] (O) to node {$At$} (M);
			
		\end{tikzpicture},
	\end{center}
if it exists, is called a \emph{globular product} of $(A,\vec{n})$. 
\end{definition}

\begin{definition}
	Let $\Theta_0$ be the category whose objects are all the tables of dimensions and 
	\[
	\Theta_0(\vec{n},\vec{m})=[\bG^\op,\Set](Y(\vec{n}),Y(\vec{m})),
	\]
    where $Y:\bG\to[\bG^\op,\Set]$ is the Yoneda embedding and $Y(\vec{n})$ and $Y(\vec{m)}$ are the colimits of the following diagrams.
    \[
    Y\circ\vec{n}:\vec{n}\to[\bG^\op,\Set]
    \]
     \[
    Y\circ\vec{m}:\vec{m}\to[\bG^\op,\Set]
    \]
\end{definition}
Let $\vec{n}=(n_1,\dots,n_{2k+1})$ be a table of dimensions and $\vec{m}$ be another tables of dimensions. Then
\[
\Theta_0^\op(\vec{m},\vec{n})=\Theta_0(\vec{n},\vec{m})=[\bG^\op,\Set](Y(\vec{n}),Y(\vec{m}))\cong \lim_{1\leq i\leq 2K+1}[\bG^\op,\Set](Y(n_i),Y(\vec{m}))
\]
\[
=\lim_{1\leq i\leq 2K+1}\Theta_0(n_i,\vec{m})=\lim_{1\leq i\leq 2K+1}\Theta_0^\op(\vec{m},n_i)=\Theta_0^\op(\vec{m},\lim_{1\leq i\leq 2K+1}n_i),
\]
so that
\[
\vec{n}\cong \lim_{1\leq i\leq 2K+1}n_i
\]
and therefore $\vec{n}$ is a globular product. Therefore $\Theta_0^\op$ has globular products. 

\begin{lemma}(Bourke, Lemma 2.1 of \cite{Bo2}) \label{univ_prop_of_tables_of_dim}
	There is a functor $D:\bG^\op\to\Theta_0^\op$ defined by $D(n)=(n)$ on objects that satisfies the following universal property: if $C$ is a category admitting $A$-globular products, there exists an essentially unique extension
	\[
	\begin{tikzpicture}[node distance=2cm, auto]
		
		\node (A) {$\bG^\op$};
		\node (B) [right of=A]{$C$} ;
		\node (C) [above of=A]{$\Theta_0^\op$};

		\draw[->,swap] (A) to node {$A$} (B);
		\draw[->](A) to node {$D^\op$}(C);
		\draw[->](C) to node {$A'$}(B);
		
	\end{tikzpicture}
	\]
	of $A$ to a globular product preserving functor $A'$. This sends $\vec{n}$ to the associated globular product.
\end{lemma}

\begin{definition}
	An \emph{extension over $\Theta_0^\op$} is a functor $H:\Theta_0^\op\to C$ that preserves globular products. A map of extensions from $H:\Theta_0^\op\to C$ to another $K:\Theta_0^\op\to D$ is a map $\gamma:C\to D$ that preserves globular products and satisfies the following equation.
	\[
	\gamma\circ H=K
	\]
	\end{definition}
	
\begin{notation}
	We write $\Ex_{\Theta_0^\op}$ to denote the category of extensions over $\Theta_0^\op$ and maps between them. Moreover, we write $\Th_{\Theta_0^\op}$ to denote the full subcategory of $\Ex_{\Theta_0^\op}$ whose objects are identity on objects functors. The objects of $\Th_{\Theta_0^\op}$ will be called \emph{globular theories}.
\end{notation}

\begin{definition}
	Let $A$ be a globular object of $C$.  An \emph{admissible pair} of $n$-cells in $A$ is a pair 
	\[
		\begin{tikzpicture}[node distance=2cm]
		\node (A) {$X$};
			\node (B)[right of=A]{$A(n)$};
	\draw[transform canvas={yshift=0.3ex},->] (A) to node[above=3] {$f$} (B);
			\draw[transform canvas={yshift=-0.3ex},->,swap](A) to node[below=3] {$g$} (B);
	\end{tikzpicture}
	\]
	of maps in $C$ where $X$ is a globular product of $A$ and $\hgt(X)\leq n+1$ such that either $n=0$ or $s\circ f=s\circ g$ and $t\circ f=t\circ g$. The \emph{dimension} of an admissible pair of the form above is $\dim(f,g)=n$. We call $(f,g)$ an admissible pair at height $n$.
\end{definition}
This is often called a parallel pair. We take the naming convention from Ara and call our pairs admissible.

\begin{definition}	
	Let $A$ be a globular object of $C$. A \emph{lift} for an admissible pair is an arrow $\delta_{f{,}g}:X\to A(n+1)$ such that the triangles in the following diagram commute serially.
	\[
		\begin{tikzpicture}[node distance=2cm]
			\node (A) {$X$};
			\node (B)[right of=A]{$A(n)$};
			\node (C) [above of=B]{$A(n+1)$};
			\draw[transform canvas={yshift=0.3ex},->] (A) to node[above=3] {$f$} (B);
			\draw[transform canvas={yshift=-0.3ex},->,swap](A) to node[below=3] {$g$} (B);
			\draw[transform canvas={xshift=-0.3ex},->] (C) to node[left=3] {$s$} (B);
			\draw[transform canvas={xshift=0.3ex},->,swap](C) to node[right=3] {$t$} (B);
			\draw[->](A) to node[above=4] {$\delta_{f,g}$}(C);
		\end{tikzpicture}
		\]
\end{definition}

\begin{definition}
	We say $A$ is \emph{contractible} if every admissable pair  has a lift .
\end{definition}

\begin{definition}
	We say that a globular theory  $J:\Theta_0^\op\to C$ is \emph{contractible} if it is contractible as a globular object of $C$.
\end{definition}

\begin{lemma}
	Given a map $F:C\to D$ of globular theories, $F$ maps admissible pairs to admissible pairs.
\end{lemma}

\begin{proof}
Suppose 
\[
		\begin{tikzpicture}[node distance=2cm]
		\node (A) {$\vec{p}$};
			\node (B)[right of=A]{$k$};
	\draw[transform canvas={yshift=0.3ex},->] (A) to node[above=3] {$f$} (B);
			\draw[transform canvas={yshift=-0.3ex},->,swap](A) to node[below=3] {$g$} (B);
	\end{tikzpicture}
	\]
is an admissible pair in $C$. Then 
\[
s\circ F(f)=F(s\circ f)=F(s\circ g)=s\circ F(g)
\]
and 
\[
t\circ F(f)=F(t\circ f)=F(t\circ g)=t\circ F(g),
\]
so that 
\[
		\begin{tikzpicture}[node distance=2cm]
		\node (A) {$\vec{p}$};
			\node (B)[right of=A]{$k$};
	\draw[transform canvas={yshift=0.3ex},->] (A) to node[above=3] {$Ff$} (B);
			\draw[transform canvas={yshift=-0.3ex},->,swap](A) to node[below=3] {$Fg$} (B);
	\end{tikzpicture}
	\]
	is admissible in $D$.
\end{proof}

\begin{definition}\label{eq:8}
	Let $J:\Theta^\op\to C$ be a globular theory. We say that $C$ is an \emph{$(\infty,0)$-coherator} if $J$ is contractible and there is a diagram of the form in the category of theories where $C$ is the colimit. 
	\[
	C_0=\Theta_0^\op\to C_1\to C_2\to\cdots\to C_n\to\cdots
	\]
	The morphism $C_n\to C_{n+1}$ is constructed by taking a set $U_n$ of admissible pairs of $C_n$ and freely forming $C_{n+1}$ from $C_n$ by formally adding a lift to each admissible pair in $U_n$ for $n\geq 0$ 
\end{definition}

\begin{notation}
Let $C$ be an $(\infty,0)$-coherator.  The category 
\[
\inftygpd_C=\Mod_{\Theta_0^\op}(C)
\]
is the full subcategory of $[C,\Set]$ containing the globular product preserving functors. The objects are called \emph{Grothendieck $\infty$-groupoids} or \emph{$\infty$-groupoids}, for short.
\end{notation}

\section{The Inductive Coherator}
In this section, we construct a cofibrantly generated algebraic weak factorization system on $\Th_{\Theta_0^\op}$. Then we partition the generating maps for the cofibrantly generated algebraic weak factorization system and obtain an $\aleph_0$-indexed collection of algebraic weak factorization systems. We show how to upgrade the induced fibrant replacement monads to a completable distributive series of monads.
\begin{definition}\label{spheres_in_inf_gpds}
	For all $\vec{p}\in\ob(\Theta_0^\op)$ and $k\geq 0$ with $\hgt(\vec{p})\leq k+1$ , define 
	\[
	S_{\vec{p},k}
	\]
	to be the theory obtained by freely adding two maps 
\[
		\begin{tikzpicture}[node distance=2cm]
			\node (A) {$\vec{p}$};
			\node (B)[right of=A]{$k$};
			\draw[transform canvas={yshift=0.3ex},->] (A) to node[above=3] {$f$} (B);
			\draw[transform canvas={yshift=-0.3ex},->,swap](A) to node[below=3] {$g$} (B);
		\end{tikzpicture}
		\]
to $\Theta_0^\op$ such that $s\circ f=s\circ g$ and $t\circ f=t\circ g$ when $k\geq 1$. This comes equipped with a theory structure map.
	\[
	\Theta_0^\op\hookrightarrow S_{\vec{p},k}
	\]
\end{definition}

\begin{definition}\label{disks_in_inf_gpds}
Let $D_{\vec{p},k}$ be obtained by freely adding a map $\delta_{f,g}:\vec{p}\to k+1$ to $S_{\vec{p},k}$  such that the triangles 

\[
		\begin{tikzpicture}[node distance=2cm]
			\node (A) {$\vec{p}$};
			\node (B)[right of=A]{$k$};
			\node (C) [above of=B]{$k+1$};
			\draw[transform canvas={yshift=0.3ex},->] (A) to node[above=3] {$f$} (B);
			\draw[transform canvas={yshift=-0.3ex},->,swap](A) to node[below=3] {$g$} (B);
			\draw[transform canvas={xshift=-0.3ex},->] (C) to node[left=3] {$s$} (B);
			\draw[transform canvas={xshift=0.3ex},->,swap](C) to node[right=3] {$t$} (B);
			\draw[->](A) to node[above=4] {$\delta_{f,g}$}(C);
		\end{tikzpicture}
		\]
commute. This comes equipped with the following theory structure map.
\[
\Theta_0^\op\hookrightarrow D_{\vec{p},k}
\]
\end{definition}
 
\begin{notation}\label{incl_of_sphere_into_disks}
There is an inclusion map of theories.
\[
j_{\vec{p},k}:S_{\vec{p},k}\to D_{\vec{p},k}
\] 
for all tables of dimensions $\vec{p}$ and $n\geq 0$. We let 
	\[
	I:=\{S_{\vec{p},k}\xrightarrow{j_{\vec{p},k}} D_{\vec{p},k}:\vec{p}\in\ob(C),k\geq 0\}.
	\]
\end{notation}
\begin{lemma} \label{admiss_1}
	The objects constructed in Definition \ref{spheres_in_inf_gpds} and Definition \ref{disks_in_inf_gpds} are presentable objects of $\Th_{\Theta_0^\op}$.  Moreover, the set $I$ of Notation \ref{incl_of_sphere_into_disks} is admissible for the ASOA of \cite{Garn2}.
\end{lemma}

\begin{proof}
See Subsection 3.11 and Lemma 3.12 of \cite{Malt}.
\end{proof}

\begin{notation}
	Let $(L_I,E_I,R_I,\delta_I)$ be the AWFS cofibrantly generated by $I$ in $\Th_{\Theta_0^\op}$. 
\end{notation}
\begin{lemma}
Let $J^{FC}:\Theta_0^\op\to FC$ be the fibrant replacement of the initial object of $\id_{\Theta_0^\op}$ with respect to the AWFS $(L_I,E_I,R_I,\delta_I)$ on $\Th_{\Theta_0^\op}$. Then $FC$ is an $(\infty,0)$-coherator.
\end{lemma}

\begin{proof}
Just rework the proof of Theorem 3.14 of \cite{Malt}.
\end{proof}

We use $FC$ to denote that $FC$ was our first idea of choice of coherator for weak $\infty$-groupoids. We note that $FC$ is the reduced coherator mentioned in Example 2.12 of \cite{Dim1}.  However, we shall not $FC$ as our choice of theory. We choose to build a coherator whose data is added inductively by the height of admissible pairs. To do so, we define
\[
I_k=\{S_{\vec{p},k}\xrightarrow{j_{\vec{p},k}} D_{\vec{p},k}:\vec{p}\in\ob(\Theta_0^\op)\}
\]
for all $k\geq 0$.

\begin{lemma}
The sets $I_k$ are admissible for the ASOA of \cite{Garn2}.
\end{lemma}
 \begin{proof}
See Lemma \ref{admiss_1}.
\end{proof}

\begin{notation}
	Let $(L_{I_k},E_{I_k},R_{I_k},\delta_{I_k})$ be the AWFS cofibrantly generated by $I_k$ in $\Th_{\Theta_0^\op}$ for all $k\geq 0$. Moreover, let $(R_{k})_{k=0}^\infty$ denote the corresponding fibrant replacement monads on $\Th_{\Theta_0^\op}$ by the AWFSs.
\end{notation}

\noindent We now name and organize the structure data.
\begin{notation}\label{fibr_repl_AWFS}
	Let $(R_{k})_{k=0}^\infty$ be the corresponding fibrant replacement monads on $\Th_{\Theta_0^\op}$ cofibrantly generated by the AWFSs. We write the following notation: $\eta^{k}:1_{\Th_{\Theta_0^\op}}\Rightarrow R_{k}$ and $\mu^{k}:R_{k}^2\Rightarrow R_{k}$ are the unit and multiplication corresponding to the monad $R_{k}$ for all $k\geq 0$, respectively.
\end{notation}

\noindent We shall make use of the following lemma repeatedly. 

\begin{lemma}\label{univ_prop_monad_unit}
	Given a theory $C$, a map $F:C\to D$ of globular theories, and given a choice of lift $\delta_{Ff,Fg}:\vec{p}\to k+1$ in $D$ for the image under $F$ of every admissible pair of the form
	\[
		\begin{tikzpicture}[node distance=2cm]
		\node (A) {$\vec{p}$};
			\node (B)[right of=A]{$k$};
	\draw[transform canvas={yshift=0.3ex},->] (A) to node[above=3] {$f$} (B);
			\draw[transform canvas={yshift=-0.3ex},->,swap](A) to node[below=3] {$g$} (B);
	\end{tikzpicture}
	\]
	in $C$, there is a unique map
	\[
	F':R_kC\to D
	\]
	such that
	\[
	F'\circ\eta^k_C=F
	\]
	and
	\[
	F'(\delta_{f,g})=\delta_{Ff,Fg}
	\]
	for every admissible pair of the form
\[
		\begin{tikzpicture}[node distance=2cm]
		\node (A) {$\vec{p}$};
			\node (B)[right of=A]{$k$};
	\draw[transform canvas={yshift=0.3ex},->] (A) to node[above=3] {$f$} (B);
			\draw[transform canvas={yshift=-0.3ex},->,swap](A) to node[below=3] {$g$} (B);
	\end{tikzpicture}
	\]
	in $C$.
\end{lemma}

The universal property of Lemma \ref{univ_prop_monad_unit} actually determines the entire monad structure. Our next goal will be to show how the fibrant replacement monads induced by the partition forms a completable distributive series of monads. 
\begin{construction} \label{distlaw_comp}
	Given a globular theory $C$, we define a map
	\[
	\lambda^{i,j}_C: R_iR_jC\Rightarrow R_jR_iC
	\]
	
	for $0\leq i<j$ as follows. Define $\lambda^{i,j}_C:R_iR_jC\to R_jR_iC$ to be the unique map such that the diagram
		\[
			\begin{tikzpicture}[node distance=2cm,auto]
				\node (A) {$R_jC$};
				\node (B)[right of=A]{$R_jR_iC$};
				\node (C)[below of=A]{$R_iR_jC$};
				\draw[->,swap] (A) to node {$\eta^i_{R_jC}$} (C);
				\draw[->](A) to node {$R_j\eta^i_C$}(B);
				\draw[->,swap] (C) to node {$\lambda^{i,j}_C$}(B);
			\end{tikzpicture}
		\]
		commutes and 
		\[
		\lambda^{i,j}_C(\delta_{f,g})=\delta_{R_j\eta^i_C(f),R_j\eta^i_C(g)}
		\]
	for every admissible pair of the form 
		\[
		\begin{tikzpicture}[node distance=2cm]
		\node (A) {$\vec{p}$};
			\node (B)[right of=A]{$i$};
	\draw[transform canvas={yshift=0.3ex},->] (A) to node[above=3] {$f$} (B);
			\draw[transform canvas={yshift=-0.3ex},->,swap](A) to node[below=3] {$g$} (B);
	\end{tikzpicture}
	\]
	in $R_j(C)$.
\end{construction}

\begin{lemma}\label{distr_tringle_2}
	For $0\leq i<j$ and every theory $C$, the triangle 
	\[
	\begin{tikzpicture}[node distance=2cm]
		\node (B){$R_iC$};
		\node (C)[below of=B]{};
		\node (D)[left of=C]{$R_iR_jC$};
		\node (E)[right of=C]{$R_jR_iC$};
		\draw[->](B) to node[left=3]{$R_i\eta^j_C$}(D);
		\draw[->](B) to node[right=3]{$\eta^j_{R_iC}$}(E);
		\draw[->](D) to node[below=3]{$\lambda^{i,j}_C$}(E);
	\end{tikzpicture}
	\]
	commutes.
\end{lemma}
\begin{proof}
The parts of the following diagram	with equal signs in the middle commute and the outside of the diagram commutes by definition and naturality. 
	\[
	\begin{tikzpicture}[node distance=2cm]
		\node (A){$C$};
		\node (B)[node distance=5cm,left of=A]{$R_iC$};
		\node (C)[below of=B]{};
		\node (D)[left of=C]{$R_iR_jC$};
		\node (E)[right of=C]{$R_jR_iC$};
		\node (F)[below of=C]{$R_jC$};
		\node (G)[node distance=1.5cm,above of=F]{$=$};
		\node (H)[node distance=1cm,right of=E]{$=$};
		\draw[->](A) to node[above=3]{$\eta^i_C$} (B);
		\draw[->,bend left=60](A) to node[right=3]{$\eta^j_C$}(F);
		\draw[->](B) to node[left=3]{$R_i\eta^j_C$}(D);
		\draw[->](B) to node[right=3]{$\eta^j_{R_iC}$}(E);
		\draw[->](D) to node[above=3]{$\lambda^{i,j}_C$}(E);
		\draw[->](F) to node[left=3]{$\eta^i_{R_jC}$}(D);
		\draw[->](F) to node[right=3]{$R_j\eta^i_C$}(E);
	\end{tikzpicture}
	\]
Therefore, we have that
	\[
	\eta^j_{R_iC}\circ \eta^i_C=R_j\eta^i_C\circ \eta^j_C
	\]
	\[
	=\lambda^{i,j}_C\circ\eta^i_{R_jC}\circ\eta^j_C=\lambda^{i,j}_C\circ R_i\eta^j_C\circ \eta^i_C.
	\]
Given an admissible pair
\[
		\begin{tikzpicture}[node distance=2cm]
		\node (A) {$\vec{p}$};
			\node (B)[right of=A]{$i$};
	\draw[transform canvas={yshift=0.3ex},->] (A) to node[above=3] {$f$} (B);
			\draw[transform canvas={yshift=-0.3ex},->,swap](A) to node[below=3] {$g$} (B);
	\end{tikzpicture}
	\]
in $R_i(C)$, we are forced to have that
\[
\lambda^{i,j}_C(R_i(\eta^j_C)(\delta_{f,g}))=\lambda^{i,j}_C(\delta_{f,g})=\delta_{f,g}=\eta^j_{R_i(C)}(\delta_{f,g}).
\]
 By the universal property of Lemma \ref{univ_prop_monad_unit}, the triangle without an equal sign must commute. 
\end{proof}

\begin{lemma}\label{dist_law_ij}
	The maps $\lambda^{i,j}_C:R_iR_jC\to R_jR_iC$ of Lemma \ref{distlaw_comp} form a natural transformation
	\[
	\lambda^{i,j}: R_iR_j\Rightarrow R_jR_i
	\]
	for $0\leq i<j$.
\end{lemma}

\begin{proof}
	Fix $0\leq i<j$ and let $F:C\to D$ be a map. Notice that we have that
	\[
	R_jR_iF\circ\lambda^{i,j}_C\circ \eta^i_{R_jC}\circ \eta^j_C=	R_jR_iF\circ\lambda^{i,j}_C\circ R_i\eta^j_C\circ \eta^i_C
	\]
	\[
	=R_jR_iF\circ \eta^j_{R_iC} \circ \eta^i_C=\eta^j_{R_iD}\circ R_iF\circ \eta^i_C
	\]
	\[
	=\lambda^{i,j}_D\circ R_i\eta^j_D\circ R_iF\circ \eta^i_C=\lambda^{i,j}_D\circ R_iR_jF\circ R_i\eta^j_C \circ \eta^i_C
	\]
	\[
	=\lambda^{i,j}_D\circ R_iR_jF\circ \eta^i_{R_jC} \circ \eta^i_C.
	\]
	\noindent Given an admissible pair 
	\[
		\begin{tikzpicture}[node distance=2cm]
		\node (A) {$\vec{p}$};
			\node (B)[right of=A]{$j$};
	\draw[transform canvas={yshift=0.3ex},->] (A) to node[above=3] {$f$} (B);
			\draw[transform canvas={yshift=-0.3ex},->,swap](A) to node[below=3] {$g$} (B);
	\end{tikzpicture}
	\]
	in $C$, we are forced to have 
	\[
	(R_jR_iF\circ\lambda^{i,j}_C\circ \eta^i_{R_jC})(\delta_{f,g})=(\lambda^{i,j}_D\circ R_iR_jF\circ \eta^i_{R_jC})(\delta_{f,g}).
	\]
	
	By the universal property of Lemma \ref{univ_prop_monad_unit}, we have that
	\[
	R_jR_iF\circ\lambda^{i,j}_C\circ \eta^i_{R_jC}=\lambda^{i,j}_D\circ R_iR_jF\circ \eta^i_{R_jC}.
	\]
Given an admissible pair 
	\[
		\begin{tikzpicture}[node distance=2cm]
		\node (A) {$\vec{p}$};
			\node (B)[right of=A]{$i$};
	\draw[transform canvas={yshift=0.3ex},->] (A) to node[above=3] {$f$} (B);
			\draw[transform canvas={yshift=-0.3ex},->,swap](A) to node[below=3] {$g$} (B);
	\end{tikzpicture}
	\]
	in $R_j(C)$, we are forced to have that
	\[
	(R_jR_iF\circ\lambda^{i,j}_C)(\delta_{f,g})=(\lambda^{i,j}_D\circ R_iR_jF)(\delta_{f,g}).
	\]
	By the universal property of Lemma \ref{univ_prop_monad_unit}, the diagram
	\[
	\begin{tikzpicture}[node distance=2cm,auto]
		\node (A) {$R_iR_jC$};
		\node (B)[right of=A] {$R_jR_iC$};
		\node(C)[below of=A] {$R_iR_jD$};
		\node (D)[right of=C] {$R_jR_iD$};
		\draw[->] (A) to node {$\lambda^{i,j}_C$}(B);
		\draw[->,swap] (A) to node{$R_iR_jF$}(C);
		\draw[->] (B) to node{$R_jR_iF$}(D);
		\draw[->,swap] (C) to node {$\lambda^{i,j}_D$}(D);
	\end{tikzpicture}
	\]
must commute. Therefore, $\lambda^{i,j}$ is a natural transformation for $0\leq i<j$.
\end{proof}

Before we continue onto the next lemma, we remark that all the diagrams that have commuted thus far have commuted because both ways of sending a choice of lift around the appropriate diagrams are the same. This is merely a consequence of the construction of the units of the monads and the rest of the structure being generated by the units. Our diagrams going forward will commute for the exact same reason. 

\begin{lemma}\label{distr_ser_fibr_repl}
	For $0\leq i<j$, $\lambda^{i,j}:R_iR_j\Rightarrow R_jR_i$ is a distributive law.	
\end{lemma}

\begin{proof}
	By Construction \ref{distlaw_comp} and Lemma \ref{distr_tringle_2}, 
\[
		\begin{tikzpicture}[node distance=2cm,auto]
			\node (A) {$R_jC$};
			\node (B)[right of=A]{$R_jR_iC$};
			\node (C)[below of=A]{$R_iR_jC$};
			\draw[->,swap] (A) to node {$\eta^i_{R_jC}$} (C);
			\draw[->](A) to node {$R_j\eta^i_C$}(B);
			\draw[->,swap] (C) to node {$\lambda^{i,j}_C$}(B);
		\end{tikzpicture}	
\]
	and 
	\[
	\begin{tikzpicture}[node distance=2cm]
		\node (B){$R_iC$};
		\node (C)[below of=B]{};
		\node (D)[left of=C]{$R_iR_jC$};
		\node (E)[right of=C]{$R_jR_iC$};
		\draw[->](B) to node[left=3]{$R_i\eta^j_C$}(D);
		\draw[->](B) to node[right=3]{$\eta^j_{R_iC}$}(E);
		\draw[->](D) to node[below=3]{$\lambda^{i,j}_C$}(E);
	\end{tikzpicture}
	\]
	both commute for every theory $C$. 
	
	\noindent We now diagram chase and obtain the equality
	\[
	\lambda^{i,j}\circ \mu^i_{R_j}\circ \eta^i_{R_iR_j}\circ\eta^i_{R_j}=\lambda^{i,j}\circ\eta^i_{R_j}=R_j\eta^i=R_j\mu^i\circ R_j\eta^i_{R_i}\circ R_j\eta^i
	\]
	\[
	=R_j\mu^i\circ\lambda^{i,j}_{R_i}\circ \eta^i_{R_jR_i}\circ R_j\eta^i=R_j\mu^i\circ\lambda^{i,j}_{R_i}\circ R_iR_j\eta^i\circ\eta^i_{R_j}
	\]
	\[
	=R_j\mu^i\circ\lambda^{i,j}_{R_i}\circ R_i\lambda^{i,j}\circ R_i\eta^i_{R_j}\circ\eta^i_{R_j}=R_j\mu^i\circ\lambda^{i,j}_{R_i}\circ R_i\lambda^{i,j}\circ \eta^i_{R_iR_j}\circ\eta^i_{R_j}
	\]
	of natural transformations. Let $C$ be a theory and let 
	\[
		\begin{tikzpicture}[node distance=2cm]
		\node (A) {$\vec{p}$};
			\node (B)[right of=A]{$i$};
	\draw[transform canvas={yshift=0.3ex},->] (A) to node[above=3] {$f$} (B);
			\draw[transform canvas={yshift=-0.3ex},->,swap](A) to node[below=3] {$g$} (B);
	\end{tikzpicture}
	\]
	be an admissible pair in $R_j(C)$. We are forced to have that
	\[
	(\lambda^{i,j}\circ \mu^i_{R_j}\circ \eta^i_{R_iR_j})_C(\delta_{f,g})=(R_j\mu^i\circ\lambda^{i,j}_{R_i}\circ R_i\lambda^{i,j}\circ \eta^i_{R_iR_j})_C(\delta_{f,g}).
	\]
	The universal property of Lemma \ref{univ_prop_monad_unit} forces
	\[
	(\lambda^{i,j}\circ \mu^i_{R_j}\circ \eta^i_{R_iR_j})_C=(R_j\mu^i\circ\lambda^{i,j}_{R_i}\circ R_i\lambda^{i,j}\circ \eta^i_{R_iR_j})_C, 
	\]
	so that 
	\[
		\lambda^{i,j}\circ \mu^i_{R_j}\circ \eta^i_{R_iR_j}=R_j\mu^i\circ\lambda^{i,j}_{R_i}\circ R_i\lambda^{i,j}\circ \eta^i_{R_iR_j}.
	\]
	Similarly, the diagram
	\[
	\begin{tikzpicture}[node distance=2cm]
		\node (E)[below of=D]{$R_i^2R_j$};
		\node (F)[right of=E]{$R_iR_j$};
		\node (G)[right of=F]{$R_jR_i$};
		\node (H)[below of=E]{$R_iR_jR_i$};
		\node (I)[node distance=4cm,right of=H,swap]{$R_jR_i^2$};
		
		\draw[->](E) to node[above=3]{$\mu^i_{R_j}$}(F);
		\draw[->](F) to node[above=3]{$\lambda^{i,j}$}(G);
		\draw[->,swap] (E) to node[left=3]{$R_i\lambda^{i,j}$}(H);
		\draw[->,swap] (H) to node[below=3]{$\lambda^{i,j}_{R_i}$}(I);
		\draw[->](I) to node[right =3]{$R_j\mu^i$}(G);
	\end{tikzpicture}
	\]
commutes by the universal property of Lemma \ref{univ_prop_monad_unit}.  By a symmetric argument, the rectangle 
	\[
	\begin{tikzpicture}[node distance=2cm]
		\node (A) {$R_iR_j^2$};
		\node (B)[below of =A]{$R_iR_j$};
		\node (F) [right of=A]{$R_jR_iR_j$};
		\node (G) [right of=F]{$R_j^2R_i$};
		\node (H) [below of=G]{$R_jR_i$};
		\draw[->,swap] (A) to node[black,left=3] {$R_i\mu^j$}(B);
		\draw[->] (A) to node [black,above=3] {$\lambda^{i,j}_{R_j}$}(F);
		\draw[->] (F) to node[black,above=3] {$R_j\lambda^{i,j}$}(G);
		\draw[->](G) to node[black,right=3] {$\mu^jR_i$}(H);
		\draw[->](B) to node[black,below=3] {$\lambda^{i,j}$}(H);
	\end{tikzpicture}
	\]
	commutes. Thus, we have proven that $\lambda^{i,j}:R_iR_j\Rightarrow R_jR_i$ is a distributive law.
\end{proof}

\begin{lemma}\label{yang_baxt_gpds}
	The Yang-Baxter equation 
	
	\[
	\begin{tikzpicture}[node distance=1.5cm]
		\node (A){$R_iR_jR_k$};
		\node (B)[right of=A]{};
		\node (C)[above of =B]{$R_jR_iR_k$};
		\node (D)[below of =B]{$R_iR_kR_j$};
		\node (E)[right of=C,node distance=3cm]{$R_jR_kR_i$};
		\node (F)[right of=D,node distance=3cm]{$R_kR_iR_j$};
		\node (G)[right of=B]{};
		\node (H)[right of =G,node distance=3cm]{$R_kR_jR_i$};
		\draw[->](A) to node[left=3]{$\lambda^{i,j}R_k$}(C);
		\draw[->,swap](A) to node[left=3]{$R_i\lambda^{j,k}$}(D);
		\draw[->](C)to node[above=3]{$R_j\lambda^{i,k}$}(E);
		\draw[->,swap](D) to node[below=3]{$\lambda^{i,k}R_j$}(F);
		\draw[->](E) to node[right=3]{$\lambda^{j,k}R_i$}(H);
		\draw[->,swap](F) to node[right=3]{$R_k\lambda^{i,j}$}(H);
	\end{tikzpicture}
	\]
	
	is satisfied for $0\leq i<j<k$.
\end{lemma}

\begin{proof}
Notice that we have 
	\[
	\lambda^{j,k}_{R_i}\circ R_j\lambda^{i,k}\circ \lambda^{i,j}_{R_k}\circ \eta^i_{R_jR_k}=\lambda^{j,k}_{R_i}\circ R_j\lambda^{i,k}\circ R_j\eta^i_{R_k}
	\]
	\[
	=\lambda^{j,k}_{R_i}\circ R_jR_k\eta^i=R_kR_j\eta^i\circ \lambda^{j,k}
	\]
	\[
	=R_k\lambda^{i,j}\circ R_k\eta^i_{R_j} \circ \lambda^{j,k}=R_k\lambda^{i,j}\circ \lambda^{i,k}_{R_j}\circ \eta^i_{R_kR_j} \circ \lambda^{j,k}
	\]
	\[
	=R_k\lambda^{i,j}\circ \lambda^{i,k}_{R_j}\circ R_i\lambda^{j,k} \circ \eta^i_{R_jR_k}.
	\]
	Let $C$ be a theory and \[
		\begin{tikzpicture}[node distance=2cm]
		\node (A) {$\vec{p}$};
			\node (B)[right of=A]{$i$};
	\draw[transform canvas={yshift=0.3ex},->] (A) to node[above=3] {$f$} (B);
			\draw[transform canvas={yshift=-0.3ex},->,swap](A) to node[below=3] {$g$} (B);
	\end{tikzpicture}
	\]
	be an admissible pair in $R_j(R_k(C))$. We are forced to have that 
\[
(\lambda^{j,k}_{R_i}\circ R_j\lambda^{i,k}\circ \lambda^{i,j}_{R_k})_C(\delta_{f,g})=(R_k\lambda^{i,j}\circ \lambda^{i,k}_{R_j}\circ R_i\lambda^{j,k})_C(\delta_{f,g}).
\]
 Upon application of the universal property of Lemma \ref{univ_prop_monad_unit}, the diagram
	\[
	\begin{tikzpicture}[node distance=1.5cm]
		\node (A){$R_iR_jR_k$};
		\node (B)[right of=A]{};
		\node (C)[above of =B]{$R_jR_iR_k$};
		\node (D)[below of =B]{$R_iR_kR_j$};
		\node (E)[right of=C,node distance=3cm]{$R_jR_kR_i$};
		\node (F)[right of=D,node distance=3cm]{$R_kR_iR_j$};
		\node (G)[right of=B]{};
		\node (H)[right of =G,node distance=3cm]{$R_kR_jR_i$};
		\draw[->](A) to node[left=3]{$\lambda^{i,j}R_k$}(C);
		\draw[->,swap](A) to node[left=3]{$R_i\lambda^{j,k}$}(D);
		\draw[->](C)to node[above=3]{$R_j\lambda^{i,k}$}(E);
		\draw[->,swap](D) to node[below=3]{$\lambda^{i,k}R_j$}(F);
		\draw[->](E) to node[right=3]{$\lambda^{j,k}R_i$}(H);
		\draw[->,swap](F) to node[right=3]{$R_k\lambda^{i,j}$}(H);
	\end{tikzpicture}
	\] 
	must commute.
\end{proof}

\begin{theorem}\label{distr_ser_AWFS}
	The fibrant replacement monads of Notation \ref{fibr_repl_AWFS} and natural transformations $\lambda^{i,j}$ of Construction \ref{distlaw_comp} form a distributive series of monads.
	\[
	R:=((R_{k})_{k=0}^\infty,(\lambda^{i,j})_{j>i\geq 0})
	\]
\end{theorem}
\begin{proof} 
Lemma \ref{dist_law_ij}, Lemma \ref{distr_ser_fibr_repl}, and Lemma \ref{yang_baxt_gpds} proves that 
\[
	R:=((R_{k})_{k=0}^\infty,(\lambda^{i,j})_{j>i\geq 0})
	\]
    constitutes a distributive series of monads.
\end{proof}

Let $(\hat{R},\hat{\eta})$ be the associated pointed endofunctor of the distributive series of Theorem \ref{distr_ser_AWFS} (see Definition \ref{ind_funct_on_dist}), which exists since $\Th_{\Theta_0^\op}$ is cocomplete. We now show the distributive series is completable. 

\begin{theorem}\label{ind_monad_AWFS}
	The structure $(\hat{R},\hat{\eta})$ extends to a monad structure on $\Th_{\Theta_0^\op}$. Moreover, the distributive series is completable. 
\end{theorem}
\begin{proof}
	We shall assign a multiplication $\hat{\mu}:\hat{R}^2\Rightarrow\hat{R}$. Define $\hat{\mu}_0:\hat{R}\to\hat{R}$ to be $\hat{\mu}_0=1_{\hat{R}}$. Suppose $C$ is a theory and that 
	\[
		\begin{tikzpicture}[node distance=2cm]
		\node (A) {$\vec{p}$};
			\node (B)[right of=A]{$0$};
	\draw[transform canvas={yshift=0.3ex},->] (A) to node[above=3] {$f$} (B);
			\draw[transform canvas={yshift=-0.3ex},->,swap](A) to node[below=3] {$g$} (B);
	\end{tikzpicture}
	\]
	is an admissible pair in $\hat{R}(C)$. There is a unique map $\hat{\mu}_1:R_0\hat{R}\to\hat{R}$ where we send the choice of lift to the admissible pair above in $R_0(\hat{R}(C))$ added along $\eta^0_{\hat{R}C}$ to the choice of lift in $\hat{R}(C)$ and such that the diagram
	\[
	\begin{tikzpicture}[node distance=3cm]
		\node (A) {$\hat{R}(C)$};
		\node (B)[below of=A]{$R_0(\hat{R}(C))$};
		\node (E) [right of=A,node distance=2cm]{$\hat{R}(C)$};
		\draw[->,swap](A) to node[black,left=3] {$\eta^0_{\hat{R}(C)}$} (B);
		\draw[->] (A) to node[black, above=3]{$1_{\hat{R}(C)}$} (E);
		\draw[->,dashed] (B) to node[right=5] {$(\hat{\mu}_1)_C$} (E);
	\end{tikzpicture}
	\]
commutes by the universal property of Lemma \ref{univ_prop_monad_unit}. Moreover, $\hat{\mu}_1$ assembles into a natural transformation. We repeat this process inductively to obtain a diagram of natural transformations.
	\[
	\begin{tikzpicture}[node distance=2cm]
		\node (A) {$\hat{R}$};
		\node (B)[below of=A]{$R_0\hat{R}$};
		\node (C)[below of=B]{$R_1R_0\hat{R}$};
		\node (D) [below of=C]{$\vdots$};
		\node (E) [right of=A,node distance=2cm]{$\hat{R}$};
		\draw[->,swap](A) to node[black,left=3] {$\eta^0_{\hat{R}}$} (B);
		\draw[->,swap](B) to node[black,left=3] {$\eta^1_{R_0\hat{R}}$} (C);
		\draw[->,swap](C) to node[black, left=3] {$\eta^2_{R_1R_0\hat{R}}$} (D);
		\draw[->] (A) to node[black, above=3]{$1_{\hat{R}}$} (E);
		\draw[->,dashed] (B) to node[right=3] {$\hat{\mu}_1$} (E);
		\draw[->,dashed,bend right=30] (C) to node[right=3] {$\hat{\mu}_2$} (E);
	\end{tikzpicture}
	\]
Upon taking colimits, we induce a unique natural transformation $\hat{\mu}:\hat{R}^2\Rightarrow\hat{R}$ such that the following diagram commutes.
	\[
	\begin{tikzpicture}[node distance=2cm]
		\node (A) {$\hat{R}$};
		\node (B)[below of=A]{$\hat{R}^2$};
		\node (E) [right of=A,node distance=2cm]{$\hat{R}$};
		\draw[->,swap](A) to node[black,left=3] {$\hat{\eta}_{\hat{R}}$} (B);
		\draw[->] (A) to node[black, above=3]{$1_{\hat{R}}$} (E);
		\draw[->,dashed] (B) to node[right=5] {$\hat{\mu}$} (E);
	\end{tikzpicture}
	\]
Notice that
	\[
	\hat{\mu}\circ \hat{\eta}_{\hat{R}}\circ \hat{\eta}=\hat{\mu}\circ \hat{R}\hat{\eta}\circ \hat{\eta}.
	\]
	By construction and inductive use of the universal property of Lemma \ref{univ_prop_monad_unit}, we must have that 
	\[
	\hat{\mu}\circ \hat{R}\hat{\eta}=\hat{\mu}\circ \hat{\eta}_{\hat{R}}=1_{\hat{R}}.
	\]
Similarly, the equation
	\[
	\hat{\mu}\circ \hat{R}\hat{\mu}\circ \hat{\eta}_{\hat{R}^2}=\hat{\mu}\circ \hat{\eta}_{\hat{R}}\circ \hat{\mu}
	\]
	\[
	=\hat{\mu}=\hat{\mu}\circ \hat{\mu}_{\hat{R}}\circ \hat{\eta}_{\hat{R}^2}
	\]
is forced to hold for the same reason.  By construction and inductive use of the universal property (see Lemma \ref{univ_prop_monad_unit}), we have that 
	\[
	\hat{\mu}\circ \hat{R}\hat{\mu}=
	\hat{\mu}\circ \hat{\mu}_{\hat{R}}.
	\]
	Therefore $(\hat{R},\hat{\eta},\hat{\mu})$ is a monad. Therefore the distributive series is completable as required.
\end{proof}

We made mention of the inductive use of the universal property. We can do this because $\hat{\eta}$ is the colimit of the diagram
\[
1_{\Th_{\Theta_0^\op}}\xrightarrow{\eta^0} R_0\xrightarrow{\eta^1_{R_0}} R_1R_0\xrightarrow{\eta^2_{R_1R_0}}\cdots
\]
in the category of endofunctors on $\Th_{\Theta_0^\op}$ and because we argue dimension-by-dimension to show that both ways of sending the appropriate choice of lift added around the required diagram are in fact the same. 
 
\begin{definition}\label{our_sp_choice}
	The \emph{inductive $(\infty,0)$-coherator} 
	\[
	J^{IC}:\Theta_0^\op\to IC
	\]
	is defined to be
	\[
	J^{IC}:=\widehat{R}(\id_{\Theta_0^\op}).
	\]
\end{definition}

\begin{corollary} \label{our_sp_choice_proof}
The theory $J^{IC}:\Theta_0^\op\to IC$ of Definition \ref{our_sp_choice} is an $(\infty,0)$-coherator.
\end{corollary}

\begin{proof}
This is an $(\infty,0)$-coherator by construction.
\end{proof}

In fact, the construction of this coherator is suggested to exist informally in the paragraph following Definition 2.5 of \cite{HenLan}. We have merely constructed a monad that gives it as an output when we put the initial globular theory in as an input.

\begin{notation}
We write $\inftygpd=\inftygpd_{IC}$.
\end{notation}

\section{Inductively Defining Weak n-Groupoids and Comparison Maps}
The inductive coherator is constructed as the colimit of the following chain in $\Th_{\Theta_0^\op}$, where $IC_{0}=\Theta_0^\op$ and $IC_{n+1}=R_n(IC_n)$ for all $n\geq 0$.
\[
IC_0\to IC_1\to\cdots\to IC_n\to IC_{n+1}\to\cdots
\]
We now begin the process of defining a chain of limit sketches that points in the opposite direction, where $IC_0^{\leq -1}$ is a one object category.
\[
\cdots \to IC_{n+2}^{\leq n+1}\to IC_{n+1}^{\leq n}\to\cdots\to IC_{1}^{\leq 0}\to IC_0^{\leq -1}
\]
This will come equipped with a map $\tr_n:IC_{n+1}\to IC_{n+1}^{\leq n}$ of limit sketches for all $n\geq 0$ such that the following square commutes for all $n\geq 0$. Moreover, we define the category of $n$-groupoids to be the category of models out of $IC_{n+1}^{\leq n}$.

\[
\begin{tikzpicture}[node distance=2cm,auto]
	\node (A) {$IC_{n+1}$};
	\node (B)[right of=A] {$IC_{n+2}$};
	\node(C)[below of=A] {$IC_{n+1}^{\leq n}$};
	\node (D)[right of=C] {$IC_{n+2}^{\leq n+1}$};
	\draw[->] (A) to node {}(B);
	\draw[->,swap] (A) to node{$\tr_n$}(C);
	\draw[->] (B) to node{$\tr_{n+1}$}(D);
	\draw[->,swap] (D) to node {$\tr$}(C);
\end{tikzpicture}
\]

\begin{notation}
Let $n\geq 0$. We write $\bG_n$ to denote the full subcategory of $\bG$ with object set $\{0,\dots,n\}$. Moreover, define $IC_{n,0}$ to be the full subcategory of $\Theta_0^\op$ whose objects consists of tables of dimensions $\vec{p}$ with $\hgt(\vec{p})\leq n$.
\end{notation}

Let $n\geq 0$. Define $\tr_n:\bG^\op\to\bG^\op_n$ by 
\[
tr_n(i)=\begin{cases}
i&\text{ if }i\leq n\\
n&\text{ otherwise}
\end{cases}
\]
Write $IC_0=\Theta_0^\op$. Define $D_n:\bG^\op\to IC_{n,0}$ by
\[
D(i)=(i).
\]
Notice that the composite 
\[
\bG^\op\xrightarrow{\tr_n}\bG^\op_n\xrightarrow{D_n}IC_{n,0}
\]
preserves globular products, so that there is an essenitally unique functor
\[
\tr_n:IC_0\to IC_{n,0}
\]
that preserves globular products such that the following square commutes.
\[
\begin{tikzpicture}[node distance=2cm,auto]
	\node (A) {$\bG^\op$};
	\node (B)[right of=A] {$\bG^\op_n$};
	\node(C)[below of=A] {$IC_0$};
	\node (D)[right of=C] {$IC_{n,0}$};
	\draw[->] (A) to node {$\tr_n$}(B);
	\draw[->,swap] (A) to node{$D$}(C);
	\draw[->] (B) to node{$D_n$}(D);
	\draw[->,swap] (C) to node {$\tr_n$}(D);
\end{tikzpicture}
\]
The functor is actually unique because the isomorphism class are just singletons. We will now form a diagram where every functor preserves globular products.
\[
\begin{tikzpicture}[>=stealth, thick, node distance=1.8cm]

\node (A0) {$IC_0$};
\node (A1) [right of=A0] {$IC_1$};
\node (A2) [right of=A1] {$\cdots$};
\node (An) [right of=A2] {$IC_n$};
\node (An1) [right of=An] {$IC_{n+1}$};

\node (B0) [below of=A0, node distance=2cm] {$IC_{n,0}$};
\node (B1) [right of=B0] {$IC_{n,1}$};
\node (B2) [right of=B1] {$\cdots$};
\node (Bn) [right of=B2] {$IC_{n,n}$};
\node (Bn1) [right of=Bn] {$IC_{n+1}^{\leq n}$};

\draw[->] (A0) to node[above] {} (A1);
\draw[->] (A1) to node[above] {} (A2);
\draw[->] (A2) to node[above] {} (An);
\draw[->] (An) to node[above] {} (An1);

\draw[->] (B0) to node[below] {} (B1);
\draw[->] (B1) to node[below] {} (B2);
\draw[->] (B2) to node[below] {} (Bn);
\draw[->] (Bn) to node[below] {} (Bn1);

\draw[->] (A0) to node[left] {$\tr_n$} (B0);
\draw[->] (A1) to node[left] {$\tr_n$} (B1);
\draw[->] (An) to node[left] {$\tr_n$} (Bn);
\draw[->] (An1) to node[left] {$\tr_n$} (Bn1);

\end{tikzpicture}
\]
We have already constructed the horizontal maps on top of this diagram: it is obtained by unraveling the data of the inductive coherator in Definition \ref{our_sp_choice}. We construct the rest of the diagram except the last step inductively. Suppose that for some $0\leq i\leq n-1$, we have obtained a map 
\[
\tr_n:IC_i\to IC_{n,i}
\]
that preserves globular products. We define the theory extension
\[
IC_{n,i}\to IC_{n,i+1}
\]
by freely adding a lift for every admissible pair at height $i$. If $(f,g)$ is such an admissible pair, we write $\delta_{f,g}$ to denote such a lift freely added.  Now let $(f,g)$ be an admissible pair of height $i$ in $IC_{i}$. Then $(\tr_n(f),\tr_n(g))$ is an  admissible pair of height $i$ in $IC_{n,i}$. We now define $\tr_n:IC_{i+1}\to IC_{n,i+1}$ to be the unique map that preserves globular products such that 
\[
\tr_n(\delta_{f,g})=\delta_{\tr_n(f),\tr_n(g)}
\]
for every admissible pair $(f,g)$ of height $i$ in $IC_{i}$ and the following diagram commutes. 
\[
\begin{tikzpicture}[node distance=2cm,auto]
	\node (A) {$IC_i$};
	\node (B)[right of=A] {$IC_{n,i}$};
	\node(C)[below of=A] {$IC_{i+1}$};
	\node (D)[right of=C] {$IC_{n,i+1}$};
	\draw[->] (A) to node {$\tr_n$}(B);
	\draw[->,swap] (A) to node{}(C);
	\draw[->] (B) to node{}(D);
	\draw[->,swap] (C) to node {$\tr_n$}(D);
\end{tikzpicture}
\]
Now we construct the final square in our desired diagram. Define the theory extension 
\[
IC_{n,n}\to IC_{n+1}^{\leq n}
\]
to be obtained by identifying every admissible pair at height $n$. We define $\tr_n:IC_{n+1}\to IC_{n+1}^{\leq n}$ to be the unique map that preserves globular products such that 
\[
\tr_n(\delta_{f,g})=[f]
\]
for every admissible pair $(f,g)$ of height $n$ in $IC_{n}$ and the following diagram commutes. 
\[
\begin{tikzpicture}[node distance=2cm,auto]
	\node (A) {$IC_n$};
	\node (B)[right of=A] {$IC_{n,n}$};
	\node(C)[below of=A] {$IC_{n+1}$};
	\node (D)[right of=C] {$IC_{n+1}^{\leq n}$};
	\draw[->] (A) to node {$\tr_n$}(B);
	\draw[->,swap] (A) to node{}(C);
	\draw[->] (B) to node{}(D);
	\draw[->,swap] (C) to node {$\tr_n$}(D);
\end{tikzpicture}
\]
We now define $t:\bG^\op_{n+1}\to\bG^\op_n$  by 
\[
\tr(i)=\begin{cases}
i&\text{ if }i\leq n\\
n&\text{ otherwise}
\end{cases}
\]
The following diagram commutes for all $n\geq 0$.
\[
\begin{tikzpicture}[>=stealth, thick, node distance=2.5cm]

\node (A) {$\bG^\op$};
\node (B) [right of=A] {$\bG^\op_{n+1}$};
\node (C) [below of=B, node distance=2.5cm] {$\bG^\op_n$};

\draw[->] (A) to node[above] {$\tr_{n+1}$} (B);
\draw[->] (A) to node[left] {$\tr_n$} (C);
\draw[->] (B) to node[right] {$\tr$} (C);

\end{tikzpicture}
\]
The composite
\[
\bG^\op_{n+1}\xrightarrow{\tr}\bG^\op_n\to IC_{n,0}
\]
preserves globular products, so there is a unique map $\tr:IC_{n+1,0}\to IC_{n,0}$ that preserves globular products such that the following square commutes. 
\[
\begin{tikzpicture}[node distance=2cm,auto]
	\node (A) {$\bG^\op_{n+1}$};
	\node (B)[right of=A] {$\bG^\op_n$};
	\node(C)[below of=A] {$IC_{n+1,0}$};
	\node (D)[right of=C] {$IC_{n,0}$};
	\draw[->] (A) to node {$\tr$}(B);
	\draw[->,swap] (A) to node{$D_{n+1}$}(C);
	\draw[->] (B) to node{$D_n$}(D);
	\draw[->,swap] (C) to node {$\tr$}(D);
\end{tikzpicture}
\]

For $0\leq i\leq n-1$, suppose there is a functor $\tr:IC_{n+1,i}\to IC_{n,i}$ that preserves globular products. Define $\tr:IC_{n+1,i+1}\to IC_{n,i+1}$ to be the unique map such that
\[
\tr(\delta_{f,g})=\delta_{\tr(f),\tr(g)}
\]
for every admissible pair at height $i$ and the following diagram commutes.
\[
\begin{tikzpicture}[node distance=2cm,auto]
	\node (A) {$IC_{n+1,i}$};
	\node (B)[right of=A] {$IC_{n,i}$};
	\node(C)[below of=A] {$IC_{n+1,i+1}$};
	\node (D)[right of=C] {$IC_{n,i+1}$};
	\draw[->] (A) to node {$\tr$}(B);
	\draw[->,swap] (A) to node{}(C);
	\draw[->] (B) to node{}(D);
	\draw[->,swap] (C) to node {$\tr$}(D);
\end{tikzpicture}
\]
By induction, we obtain a map $\tr:IC_{n+1,n}\to IC_{n,n}$ that preserves globular products. Define $\tr: IC_{n+1,n+1}\to IC_{n+1}^{\leq n}$ to be the unique map such that
\[
\tr(\delta_{f,g})=[f]
\]
for every admissible pair $(f,g)$ at height $n$ and the following diagram commutes.
\[
\begin{tikzpicture}[node distance=2cm,auto]
	\node (A) {$IC_{n+1,n}$};
	\node (B)[right of=A] {$IC_{n,n}$};
	\node(C)[below of=A] {$IC_{n+1,n+1}$};
	\node (D)[right of=C] {$IC_{n+1}^{\leq n}$};
	\draw[->] (A) to node {$\tr$}(B);
	\draw[->,swap] (A) to node{}(C);
	\draw[->] (B) to node{}(D);
	\draw[->,swap] (C) to node {$\tr$}(D);
\end{tikzpicture}
\]

Notice that if $(f,g)$ is an admissible pair of height $n+1$ in $IC_{n+1,n+1}$, then $\tr(f)=\tr(g)$, so there is a unique map $\tr:IC_{n+2}^{\leq n+1}\to IC_{n+1}^{\leq n}$ that preserves globular products such that the diagram commutes.

\[
\begin{tikzpicture}[>=stealth, thick, node distance=2.5cm]

\node (A) {$IC_{n+1,n+1}$};
\node (B) [right of=A] {$IC_{n+2}^{\leq n+1}$};
\node (C) [below of=B, node distance=2.5cm] {$IC_{n+1}^{\leq n}$};

\draw[->] (A) to node[above] {} (B);
\draw[->] (A) to node[left] {$\tr$} (C);
\draw[->] (B) to node[right] {$\tr$} (C);
\end{tikzpicture}
\]

By construction, the following diagram commutes.
\[
\begin{tikzpicture}[node distance=2cm,auto]
	\node (A) {$IC_{n+1}$};
	\node (B)[right of=A] {$IC_{n+2}$};
	\node(C)[below of=A] {$IC_{n+1}^{\leq n}$};
	\node (D)[right of=C] {$IC_{n+2}^{\leq n+1}$};
	\draw[->] (A) to node {}(B);
	\draw[->,swap] (A) to node{$\tr_n$}(C);
	\draw[->] (B) to node{$\tr_{n+1}$}(D);
	\draw[->,swap] (D) to node {$\tr$}(C);
\end{tikzpicture}
\]

This induces a square of adjunctions where the outside and inside square commutes and consists of the right adjoints.
\[
\begin{tikzpicture}[>=stealth, thick, node distance=4cm]

\node (A) {$\Mod(IC_{n+2})$};
\node (B) [right of=A] {$\Mod(IC_{n+1})$};
\node (C) [below of=A] {$\Mod(IC_{n+2}^{\leq n+1})$};
\node (D) [right of=C] {$\Mod(IC_{n+1}^{\leq n})$};

\draw[transform canvas={yshift=0.3ex},->] (A) to node[above] {$U$} (B);
\draw[transform canvas={yshift=-0.3ex},<-] (A) to node[below] {$\free$} (B);

\draw[transform canvas={yshift=-0.3ex},->] (D) to node[below] {$D$} (C);
\draw[transform canvas={yshift=0.3ex},<-] (D) to node[above] {$\tr$} (C);

\draw[transform canvas={xshift=0.3ex},->] (A) to node[left] {$D$} (C);
\draw[transform canvas={xshift=-0.3ex},<-] (A) to node[right] {$\tr$} (C);

\draw[transform canvas={xshift=-0.3ex},->] (B) to node[left] {$\tr$} (D);
\draw[transform canvas={xshift=0.3ex},<-] (B) to node[right] {$D$} (D);

\end{tikzpicture}
\]
We use the letter $D$ for the right adjoint on the bottom because it sends a model $X:IC_{n+1}^{\leq n}\to\Set$ to a model
$D(X):IC_{n+2}^{\leq n+1}\to\Set$ with 
\[
D(X)_i=X_i
\]
for $0\leq i\leq n$ and
\[
D(X)(a,b)=*
\]
for all $a,b\in X_n$. We call the right adjoints on the sides $D$ for an identical reason. The right adjoint on the top is obtained by forgetting the axioms obtained by transferring along the inductive coherator on the underlying theories. The left adjoint on the top is defined by freely adding $(n+1)$-cells with respect to freely adding lifts along the chain that defines the coherator. The left adjoints are just the truncation functors.

\begin{lemma}
For all $n\geq 0$, the following diagrams commute.
\[
\begin{tikzpicture}[>=stealth, thick, node distance=3cm]

\node (A) {$\Mod(IC_{n+1}^{\leq n})$};
\node (B) [right of=A] {$\Mod(IC_{n+1})$};
\node (C) [below of=B, node distance=2.5cm] {$\Mod(IC_{n+1}^{\leq n}$};

\draw[->] (A) to node[above] {$D$} (B);
\draw[->] (A) to node[left] {$1_{\Mod(IC_{n+1}^{\leq n})}$} (C);
\draw[->] (B) to node[right] {$\tr$} (C);
\end{tikzpicture}
\]
\end{lemma}

We now set some notation. 

\begin{notation}
We write the following for all $n\geq 0$.
\[
\ngpd=\Mod(IC_{n+1}^{\leq n})
\]
\end{notation}
One thing that immediate is that there is an isomorphism of categories. 
\[
\mathpzc{0Gpd}\cong\Set
\]
We can say even more: this provides a square of adjoints

\[
\begin{tikzpicture}[>=stealth, thick, node distance=4cm]

\node (A) {$\Mod(IC_{n+2})$};
\node (B) [right of=A] {$\Mod(IC_{n+1})$};
\node (C) [below of=A] {$\mathpzc{(n+1)Gpd}$};
\node (D) [right of=C] {$\ngpd$};

\draw[transform canvas={yshift=0.3ex},->] (A) to node[above] {$U$} (B);
\draw[transform canvas={yshift=-0.3ex},<-] (A) to node[below] {$\free$} (B);

\draw[transform canvas={yshift=-0.3ex},->] (D) to node[below] {$D$} (C);
\draw[transform canvas={yshift=0.3ex},<-] (D) to node[above] {$\tr$} (C);

\draw[transform canvas={xshift=0.3ex},->] (A) to node[left] {$D$} (C);
\draw[transform canvas={xshift=-0.3ex},<-] (A) to node[right] {$\tr$} (C);

\draw[transform canvas={xshift=-0.3ex},->] (B) to node[left] {$\tr$} (D);
\draw[transform canvas={xshift=0.3ex},<-] (B) to node[right] {$D$} (D);

\end{tikzpicture}
\]
and allows us a subtle interpretation of the following square for all $n\geq 0$.

\[
\begin{tikzpicture}[node distance=2cm,auto]
	\node (A) {$IC_{n+1}$};
	\node (B)[right of=A] {$IC_{n+2}$};
	\node(C)[below of=A] {$IC_{n+1}^{\leq n}$};
	\node (D)[right of=C] {$IC_{n+2}^{\leq n+1}$};
	\draw[->] (A) to node {}(B);
	\draw[->,swap] (A) to node{$\tr_n$}(C);
	\draw[->] (B) to node{$\tr_{n+1}$}(D);
	\draw[->,swap] (D) to node {$\tr$}(C);
\end{tikzpicture}
\]

The coherator has all the necessary coherence data to encode the data for an $n$-groupoid. The quotient map (the left vertical map) encodes the axioms on the top dimension, i.e. dimension $n$. The top horizontal map corresponds to categorifying the top level axioms for an $n$-groupoid to add the coherence data in dimension $n+1$ to encode the data for an $(n+1)$-groupoid. The quotient map (the right vertical map) encodes the axioms on the top dimension, i.e. dimension $n+1$. The truncation map, i.e the bottom horizontal map, truncates the coherence data and axioms for an $(n+1)$-groupoid to the coherence data and axioms for an $n$-groupoid.

Now let us consider the chain 
\[
\cdots \xrightarrow{\tr} IC_{n+2}^{\leq n+1}\xrightarrow{\tr}IC_{n+1}^{\leq n}\xrightarrow{\tr}\cdots\xrightarrow{\tr} IC_{2}^{\leq 1}\xrightarrow{\tr}IC_1^{\leq 0}
\]
in the category of realized limit sketches. Let us write $IC_\infty^{\leq \infty}$ denote the limit. By abuse of notation, we will use $\inftygpd$ and $\Mod(IC_\infty^{\leq \infty})$ interchangeably. The following lemma says this causes no issues.

\begin{lemma}
There is an isomorphism of categories.
\[
\inftygpd\cong \Mod(IC_\infty^{\leq \infty})
\]
\end{lemma}

\section{Generating Cofibrations and Trivial Cofibrations}
We now explain the maps that we will demand to be the generating cofibrations and trivial cofibrations for the model structures in a Section 10. We first define our special objects and maps for globular sets before translating across left adjoints.

\begin{notation}
Let $k\geq 0$. We write $S^{k-1}$ to be the following globular set.
\[
S^{k-1}(i)=\begin{cases}
    \{*_s,*_t\}&\text{ if }i\leq k-1\\
    \emptyset&\text{ otherwise.}
\end{cases}
\]
\end{notation}

\begin{notation}\label{gen_cof_on_glob_set}
We write $j_k:S^{k-1}\to D^k$ to denote the natural inclusion of $S^{k-1}$ into $D^k$ for all $k\geq 0$. We write 
\[
\zI=\{S^{k-1}\xrightarrow{j_k}D^k:k\geq 0\}.
\]
\end{notation}

\begin{notation}\label{gen_triv_cof_on_glob_set}
We write $\sigma_k:D^{k}\to D^{k+1}$ to denote the inclusion of $D^k$ into $D^{k+1}$ as the source $k$-cell of the $(k+1)$-cell for all $k\geq 0$. We write 
\[
\zJ=\{D^{k}\xrightarrow{\sigma_k}D^{k+1}:k\geq 0\}.
\]
\end{notation}

These maps could be obtained by considering the image of the map
\[
s:k\to k+1
\]
under the Yoneda embedding $Y:\bG\to[\bG^\op,\Set]$ for all $k\geq 0$. We will write $D:=Y$. Notice that there is an isomorphism of globular sets, natural in $X$.
\[
\phi_X:X\to \GSet(D(-),X)
\]

As there is an equivalence of categories
\[
[\bG^\op,\Set]\simeq \Mod(IC_0)
\]
induced by the canonical map
\[
D:\bG^\op\to IC_0,
\]
we may translate the set of maps of Notation \ref{gen_cof_on_glob_set} and Notation \ref{gen_triv_cof_on_glob_set} in $[\bG^\op,\Set]$ to sets of maps  
\[
\zI=\{S^{k-1}\xrightarrow{j_k}D^k:k\geq 0\}
\]
and 
\[
\zJ=\{D^{k}\xrightarrow{\sigma_k}D^{k+1}:k\geq 0\}
\]
in $\Mod(IC_0)$.

We now inductively transfer these maps along the following chain of left adjoints.
\[
\Mod(IC_0)\xrightarrow{\free}\Mod(IC_1)\xrightarrow{\free}\cdots\xrightarrow{\free}\Mod(IC_n)\xrightarrow{\free}\Mod(IC_{n+1})\xrightarrow{\free}\cdots
\]
Inductively define the objects $S^{k-1}=\free(S^{k-1})$ and $D^k=\free(D^k)$  in $\Mod(IC_{n+1})$ for all $k,n\geq 0$. Inductively define the maps 
\[
j_k:S^{k-1}\to D^k
\]
\[
\sigma_k:D^k\to D^{k+1}
\]
by
\[
j_k=\free(j_k)
\]
\[
\sigma_k=\free(\sigma_k)
\]
in $\Mod(IC_{n+1})$ for all $k,n\geq 0$. Set 
\[
\zI'=\{S^{k-1}\xrightarrow{j_k}D^k:k\geq 0\}
\]
and 
\[
\zJ'=\{D^{k}\xrightarrow{\sigma_k}D^{k+1}:k\geq 0\}.
\] 
Upon taking colimits, we induce sets of maps 
\[
\zI'=\{S^{k-1}\xrightarrow{j_k}D^k:k\geq 0\}
\]
and 
\[
\zJ'=\{D^{k}\xrightarrow{\sigma_k}D^{k+1}:k\geq 0\}
\]
in $\inftygpd$.

\begin{notation}\label{gen cofibrations and triv cofibrations}
For all $n\geq 1$, we call $\zI'$ and $\zJ'$ the set of generating cofibrations and the set of generating trivial cofibrations in $\Mod(IC_n)$, respectively.  We call  $\zI'$ and $\zJ'$ the set of generating cofibrations and the set of generating trivial cofibrations in $\inftygpd$, respectively.
\end{notation}

For all $n\geq 1$, we define $S^{k-1}=\tr(S^{k-1})$ and $D^k=\tr(D^k)$ in $\ngpd$. We define the maps 
\[
j_k:S^{k-1}\to D^k
\]
\[
\sigma_k:D^k\to D^{k+1}
\]
by
\[
j_k=\tr(j_k)
\]
\[
\sigma_k=\tr(\sigma_k)
\]
in $\ngpd$ for all $k,n\geq 0$. Set 
\[
\zI=\{S^{k-1}\xrightarrow{j_k}D^k:k\geq 0\}
\]
and 
\[
\zJ=\{D^{k}\xrightarrow{\sigma_k}D^{k+1}:k\geq 0\}.
\] 

\begin{notation}\label{gen cofibrations and triv cofibrations_n_gpd}
For all $n\geq 1$, we call $\zI$ and $\zJ$ the set of generating cofibrations and the set of generating trivial cofibrations in $\ngpd$, respectively.
\end{notation}

Notice that if $k\geq n$, then $\sigma_k=1_{D^n}$ in $\ngpd$. If $k=n+1$, then $j_k$ is the collapse map of $S^n$ onto $D^n$. If $k>n+1$, then $j_k=1_{D^n}$.

\section{Homotopy Groups}
We now provide background on the path components and homotopy group functors for $n$-groupoids for all finite $n$. All the material presented here with the exception of Lemma \ref{path_components_transfer}, Lemma \ref{first_hom_group_transfer}, and Lemma \ref{higher_hom_group_transfer} are translated from results from \cite{Dim1}.

\begin{definition}
	Let  $n\geq 1$, $0\leq k<n$, $X$ be an $n$-groupoid, and $u,v\in X_k$. A \emph{homotopy} from
	$u$ to $v$ is an element $\alpha\in X_{k+1}$ such that 
	\[
	s(\alpha)=u,t(\alpha)=v.
	\] If such a homotopy
	exists, we will say that $u$ is \emph{homotopic} to $v$ and we will write $u
	\sim_k v$.  
\end{definition}

\begin{definition}
	Let  $n\geq 1$, $0\leq k<n$, and $X$ be an $n$-groupoid. We will denote by $\ol{X_k}$ the quotient of $X_k$ by
	the equivalence relation $\sim_n$.
\end{definition}

\begin{construction}
	Let $n\geq 1$ and $X$ be an $n$-groupoid. Define the set $\pi_0(X)$ of \emph{connected components} of $X$ to be the set
	\[ \pi_0(X) = \pi_0(\ol{\omega_1}(X)) = \overline{X_0}. \]
\end{construction}

This construction extends to a functor
\[
\pi_0:\ngpd\to \Set.
\]
Define the functor $\pi_0:\Mod(IC_{n+1})\to\Set$ to be the composite
\[
\Mod(IC_{n+1})\xrightarrow{\tr}\ngpd\xrightarrow{\pi_0}\Set
\]
The following lemmas are immediate and follow by the preservation of cells by the appropriate functors below dimension $n$.

\begin{lemma}\label{path_components_transfer}
The following diagrams commute.
\[
\begin{tikzpicture}[>=stealth, thick, node distance=3cm]

\node (A1) at (0,0) {$D^0/\mathpzc{(n+1)Gpd}$};
\node (B1) [right of=A1] {$D^0/\ngpd$};
\node (C1) [below of=B1, node distance=2.5cm] {$\Set$};

\draw[->] (A1) to node[above] {$\tr$} (B1);
\draw[->] (A1) to node[left] {$\pi_0$} (C1);
\draw[->] (B1) to node[right] {$\pi_0$} (C1);

\begin{scope}[xshift=6.5cm]

\node (A2) at (0,0) {$D^0/\Mod(IC_{n+1})$};
\node (B2) [right of=A2, node distance=4cm] {$D^0/\Mod(IC_{n+2})$};
\node (C2) [below of=B2, node distance=2.5cm] {$\Set$};

\draw[->] (A2) to node[above] {$\free$} (B2);
\draw[->] (A2) to node[left] {$\pi_0$} (C2);
\draw[->] (B2) to node[right] {$\pi_0$} (C2);

\end{scope}

\end{tikzpicture}
\]

\[
\begin{tikzpicture}[>=stealth, thick, node distance=4cm]

\node (A2) at (0,0) {$D^0/\ngpd$};
\node (B2) [right of=A2] {$D^0/\mathpzc{(n+1)Gpd}$};
\node (C2) [below of=B2, node distance=2.5cm] {$\Set$};

\draw[->] (A2) to node[above] {$D$} (B2);
\draw[->] (A2) to node[left] {$\pi_0$} (C2);
\draw[->] (B2) to node[right] {$\pi_0$} (C2);

\begin{scope}[xshift=7cm]

\node (A1) at (0,0) {$D^0/\Mod(IC_{n+2})$};
\node (B1) [right of=A1] {$D^0/\Mod(IC_{n+1})$};
\node (C1) [below of=B1, node distance=2.5cm] {$\Set$};

\draw[->] (A1) to node[above] {$U$} (B1);
\draw[->] (A1) to node[left] {$\pi_0$} (C1);
\draw[->] (B1) to node[right] {$\pi_0$} (C1);

\end{scope}

\end{tikzpicture}
\]
\end{lemma}

\begin{lemma}
	For every $n\geq 1$, $0\leq k<n$ and $n$-groupoid $X$, the relation $\sim_k$ is an equivalence relation.
	For  $n\geq 2$ and $1\leq k<n$, this relation is compatible with the composition
	$c^k_{k-1}:X_k\times_{k-1}X_k\to X_k$, i.e there is an induced function
    \[
c_{k-1}^k:\overline{X_k} \times_{k-1} \overline{X_k} \to
	\overline{X_k}.
    \]
\end{lemma}

\begin{construction}\label{groupoid_on_n_level}
	Fix $n\geq 2$ and $1\leq k<n$. The maps
	\[
	s, t : X_k \to X_{k-1},
	\quad
	Z =X(Z): X_{k-1} \to X_k
	\]
	 induce maps
	\[
	s,t : \overline{X_k} \to X_{k-1},
	\quad
	Z : X_{k-1} \to \overline{X_k}.
	\]
	By the previous lemma, the map
	\[ c^k_{k-1} : X_k \times_{k-1} X_k \to X_k \]
	induces a map
	\[
	c^k_{k-1} : \overline{X_k} \times_{k-1} \overline{X_k} \to
	\overline{X_k}.
	\]
	We use $\ol{\omega_k}(X)$ to denote the graph
	\[
	\xymatrix{
		\overline{X_k} \ar@<.6ex>[r]^-{s} \ar@<-.6ex>[r]_-{t} &
		X_{k-1},
	}
	\]
	endowed together with the maps
	\[
	c^k_{k-1} : \overline{X_k} \times_{k-1} \overline{X_k} \to
	\overline{X_k}
	\quad\text{and}\quad
	Z : X_{k-1} \to \overline{X_k}.
	\]
\end{construction}

\begin{proposition}
	For every $n\geq 2$ and $1\leq k<n$, $\ol{\omega_k}(X)$ is a groupoid.
\end{proposition}

\begin{proof}
This is Proposition 4.5 of \cite{Dim1}.
\end{proof}

\begin{proposition}
	Construction \ref{groupoid_on_n_level} extends to a functor
    \[
    \ol{\omega_k}:\ngpd\to\Gpd
    \]
    for all $n\geq 2$ and $1\leq k<n$.
\end{proposition}

\begin{proof}
	Let $f:X\to Y$ be a morphism of $n$-groupoids. Such a
	morphism induces a morphism of globular sets between the underlying $n$-globular
	sets and in particular respects the notion of homotopy between $k$-arrows.
	It follows that for any $n\geq 2$ and $1\leq k<n$, $f$ induces a morphism of graphs
	\[ \ol{\omega_k}(f):\ol{\omega_k}(X)\to\ol{\omega_k}(Y). \]
	The naturality of $f$ implies that $\ol{\omega_k}(f)$ is actually a functor. Therefore, we obtain a functor
	\[
	\ol{\omega_k}:\ngpd \to \Gpd.
	\]
\end{proof}

\begin{construction}
	Let  $n\geq 2$, $1\leq k<n$, $X$ be an $n$-groupoid and let $u, v$ be two parallel $(k-1)$-arrows of $X$. We will denote by
	$\Hom_X(u, v)$ the set of $k$-arrows of $X$ from $u$ to $v$. We set
	\[
	\pi_k(X, u, v) = \Hom_{\ol{\omega_k}(X)}(u, v)
	\quad\text{and}\quad
	\pi_k(X, u) = \pi_k(X, u, u).
	\]
\end{construction}

	The set $\pi_k(X, u, v)$ is nothing but the quotient of $\Hom_X(u, v)$ by
	the equivalence relation $\sim_k$. Note that $\pi_k(X, u)$ is
	endowed with a group structure.

\begin{construction}
	If $n\geq 2$, $1\leq k<n$, $X$ be an $n$-groupoid and $x\in X_0$, we define the \emph{$k$-th
		homotopy group} of $(X, x)$ as
	\[ \pi_k(X, x) = \pi_k(X,Z^{k-1}(x)). \]
\end{construction}

\begin{construction}
Let $n\geq 1$, $X$ be an $n$-groupoid, and $u,v\in X_{n-1}$. Define 
\[
\pi_{n}(X,u,v)=\{\alpha\in X_n:s(\alpha)=u, t(\alpha)=v\}\text{ and }\pi_n(X,u)=\pi_n(X,u,u)
\]
\end{construction}

\begin{construction}
   Let $n\geq 1$, $X$ be an $n$-groupoid, and $x\in X_{0}$. Define the $n$th homotopy group of $(X,x)$ as
   \[
   \pi_n(X,x)=\pi_n(X,Z^{n-1}(X))
   \]
\end{construction}

\begin{remark}
	The Eckmann-Hilton argument shows that for $n\geq 2$ and $2\leq k\leq n$, the group
	$\pi_k(X, x)$ is abelian. Moreover, using that  $\ol{\omega_k}$ is a functor from the
	category of $n$-groupoids to groupoids for $k\geq 1$, we get that
	\begin{itemize}
        \item $\pi_1$ is a a functor from the category of
		$n$-groupoids endowed with a $0$-cell
		to the category of groups, and 
		\item $\pi_n$ is a functor from the category of
		$n$-groupoids endowed with a $0$-cell
		to the category of abelian groups.
	\end{itemize}
    This is all from Section 4.11 of \cite{Dim1} but translated to $n$-groupoids.
\end{remark}

This means that for $n\ge 1$, we have a functor
\[
\pi_1:D^0/\ngpd\to\Grp.
\]

For $n\geq 1$, we define $\pi_1:D^0/\Mod(IC_{n+1})\to\Grp$ to be the composite
\[
D^0/\Mod(IC_{n+1})\xrightarrow{\tr^*}D^0/\ngpd\xrightarrow{\pi_1}\Grp
\]

\begin{lemma}\label{first_hom_group_transfer}
Let $n\geq 1$. The following diagrams commute.

\[
\begin{tikzpicture}[>=stealth, thick, node distance=3cm]

\node (A1) at (0,0) {$D^0/\mathpzc{(n+1)Gpd}$};
\node (B1) [right of=A1] {$D^0/\ngpd$};
\node (C1) [below of=B1, node distance=2.5cm] {$\Grp$};

\draw[->] (A1) to node[above] {$\tr$} (B1);
\draw[->] (A1) to node[left] {$\pi_1$} (C1);
\draw[->] (B1) to node[right] {$\pi_1$} (C1);

\begin{scope}[xshift=6.5cm]

\node (A2) at (0,0) {$D^0/\Mod(IC_{n+1})$};
\node (B2) [right of=A2, node distance=4cm] {$D^0/\Mod(IC_{n+2})$};
\node (C2) [below of=B2, node distance=2.5cm] {$\Grp$};

\draw[->] (A2) to node[above] {$\free$} (B2);
\draw[->] (A2) to node[left] {$\pi_1$} (C2);
\draw[->] (B2) to node[right] {$\pi_1$} (C2);

\end{scope}

\end{tikzpicture}
\]

\[
\begin{tikzpicture}[>=stealth, thick, node distance=4cm]

\node (A2) at (0,0) {$D^0/\ngpd$};
\node (B2) [right of=A2] {$D^0/\mathpzc{(n+1)Gpd}$};
\node (C2) [below of=B2, node distance=2.5cm] {$\Grp$};

\draw[->] (A2) to node[above] {$D$} (B2);
\draw[->] (A2) to node[left] {$\pi_1$} (C2);
\draw[->] (B2) to node[right] {$\pi_1$} (C2);

\begin{scope}[xshift=7cm]

\node (A1) at (0,0) {$D^0/\Mod(IC_{n+2})$};
\node (B1) [right of=A1] {$D^0/\Mod(IC_{n+1})$};
\node (C1) [below of=B1, node distance=2.5cm] {$\Grp$};

\draw[->] (A1) to node[above] {$U$} (B1);
\draw[->] (A1) to node[left] {$\pi_1$} (C1);
\draw[->] (B1) to node[right] {$\pi_1$} (C1);

\end{scope}

\end{tikzpicture}
\]
\end{lemma}

For $n\geq 2$ and $2\leq k\leq n$, we have a functor
\[
\pi_k:D^0/\ngpd\to \Ab.
\]
For $n\geq 2$ and $2\leq k\leq n$, we define $\pi_k:D^0/\Mod(IC_{n+1})\to\Grp$ to be the composite
\[
D^0/\Mod(IC_{n+1})\xrightarrow{\tr^*}D^0/\ngpd\xrightarrow{\pi_k}\Grp
\]

\begin{lemma}\label{higher_hom_group_transfer}
Let $n\geq 2$. The following diagrams commute.

\[
\begin{tikzpicture}[>=stealth, thick, node distance=3cm]

\node (A1) at (0,0) {$D^0/\mathpzc{(n+1)Gpd}$};
\node (B1) [right of=A1] {$D^0/\ngpd$};
\node (C1) [below of=B1, node distance=2.5cm] {$\Ab$};

\draw[->] (A1) to node[above] {$\tr$} (B1);
\draw[->] (A1) to node[left] {$\pi_k$} (C1);
\draw[->] (B1) to node[right] {$\pi_k$} (C1);

\begin{scope}[xshift=6.5cm]

\node (A2) at (0,0) {$D^0/\Mod(IC_{n+1})$};
\node (B2) [right of=A2, node distance=4cm] {$D^0/\Mod(IC_{n+2})$};
\node (C2) [below of=B2, node distance=2.5cm] {$\Grp$};

\draw[->] (A2) to node[above] {$\free$} (B2);
\draw[->] (A2) to node[left] {$\pi_k$} (C2);
\draw[->] (B2) to node[right] {$\pi_k$} (C2);

\end{scope}

\end{tikzpicture}
\]

\[
\begin{tikzpicture}[>=stealth, thick, node distance=4cm]

\node (A2) at (0,0) {$D^0/\ngpd$};
\node (B2) [right of=A2] {$D^0/\mathpzc{(n+1)Gpd}$};
\node (C2) [below of=B2, node distance=2.5cm] {$\Ab$};

\draw[->] (A2) to node[above] {$D$} (B2);
\draw[->] (A2) to node[left] {$\pi_k$} (C2);
\draw[->] (B2) to node[right] {$\pi_k$} (C2);

\begin{scope}[xshift=7cm]

\node (A1) at (0,0) {$D^0/\Mod(IC_{n+2})$};
\node (B1) [right of=A1] {$D^0/\Mod(IC_{n+1})$};
\node (C1) [below of=B1, node distance=2.5cm] {$\Ab$};

\draw[->] (A1) to node[above] {$U$} (B1);
\draw[->] (A1) to node[left] {$\pi_k$} (C1);
\draw[->] (B1) to node[right] {$\pi_k$} (C1);

\end{scope}

\end{tikzpicture}
\]
\end{lemma}

The previous two lemmas are true for the same reason Lemma \ref{path_components_transfer} is true. 

\section{Weak Equivalences}
We now provide background on weak equivalences. Once again, we translate the material presented in this section from \cite{Dim1} to $n$-groupoids. The only new material here, where new is used with large quotation marks, is Definition \ref{The Pushout Condition} and Definition \ref{The Free Pushout Condition}. We say in large quotation marks because the definitions help us formulate the pushout conjecture in terms of certain definitions being satisfied for every $\infty$-groupoid.
\begin{definition}
Let $n\geq 0$. We say that a map $f:X\to Y$ of $n$-groupoids is a \emph{weak equivalence} if 
\begin{itemize}
    \item $\pi_0(f)$ is a bijection on path components and 
    \item if $\pi_k(f):\pi_k(X,x)\to \pi_k(Y,f(x))$ is an isomorphism of groups for $1\leq k\leq n$ and $x\in X_0$ when $n\geq 1$.
\end{itemize}
\end{definition}

\begin{notation}\label{weak_eq_of_n_gpds}
Let $n\geq 0$. We write $\zW$ be the collection of all weak equivalences in $\ngpd$.
\end{notation}

\begin{lemma}
Let $n\geq 1$. If $f:X\to Y$ is a weak equivalence of $(n+1)$-groupoids. Then $\tr(f)$ is a weak equivalence of $n$-groupoids. Moreover, if $f:X\to Y$ is a weak equivalence of $n$-groupoids, then $D(f)$ is a weak equivalence of $(n+1)$-groupoids.
\end{lemma}

\begin{proof}
We combine Lemma \ref{path_components_transfer}, Lemma \ref{first_hom_group_transfer}, and Lemma \ref{higher_hom_group_transfer} to conclude the proof.
\end{proof}

\begin{definition}
Let $n\geq 1$. We say a map $f:X\to Y$ in $\Mod(IC_{n+1})$ is a weak equivalence if $\tr(f)$ is a weak equivalence of $n$-groupoids. 
\end{definition}

\begin{notation}\label{weak_eq_of_models}
Let $n\geq 1$. We write $\zW'$ to be the collection of all weak equivalences in $\Mod(IC_{n+1})$.
\end{notation}

\begin{lemma}
    Let $n\geq 1$. If $f:X\to Y$ is a weak equivalence in $\Mod(IC_{n+2})$, then $Uf$ is a weak equivalence in  $\Mod(IC_{n+2})$. 
\end{lemma}
\begin{proof}
We combine Lemma \ref{path_components_transfer}, Lemma \ref{first_hom_group_transfer}, and Lemma \ref{higher_hom_group_transfer} to conclude the proof.
\end{proof}

\begin{definition}\label{The Pushout Condition}
Let $n\geq 1$ and $X:IC_{n+1}\to\Set$ be a model. We say $X$ satisfies the \emph{pushout condition} if the the following holds: for all $k\geq 0$ and maps $f:D^k\to X$, the map $p:X\to X_+$ in the following pushout is a weak equivalence. 
\[
		\begin{tikzpicture}[node distance=2.5cm]
			\node(A) {$D^k$};
			\node (B) [right of=A]{$X$};
			\node(C)[below of=A]{$D^{k+1}$};
			\node (D)[right of=C]{$X_+$};
			\draw[->](B) to node[black,right=3]{$p$}(D);
			\draw[->](A) to node {}(B);
			\draw[->] (A) to node {}(C);
			\draw[->](C) to node {}(D);
		\end{tikzpicture}
		\]
\end{definition}

This definition says that the pushout conjecture is true for a particular $\infty$-groupoid.

\begin{definition}\label{The Free Pushout Condition}
Let $n\geq 1$ and $X:IC_{n+1}\to\Set$ be a model. We say $X$ satisfies the \emph{free pushout condition} if the the following holds: for all $k\geq 0$ and maps $f:D^k\to X$, the map $\free(p):\free X\to \free X_+$ is a weak equivalence in $\Mod(IC_{n+2})$, where $p:X\to X_+$ is the structure map in the following pushout. 
\[
		\begin{tikzpicture}[node distance=2.5cm]
			\node(A) {$D^k$};
			\node (B) [right of=A]{$X$};
			\node(C)[below of=A]{$D^{k+1}$};
			\node (D)[right of=C]{$X_+$};
			\draw[->](B) to node[black,right=3]{$p$}(D);
			\draw[->](A) to node {}(B);
			\draw[->] (A) to node {}(C);
			\draw[->](C) to node {}(D);
		\end{tikzpicture}
		\]
\end{definition}

This definition is a transfer variant of the pushout conjecture. In particular, we use this definition to argue about being able to hypothetically transfer model structure from the category of $n$-groupoids to the category of $(n+1)$-groupoids.

\section{The Model Structures in Finite Dimensions}
We now provide a description for the hypothetical model structures on $\ngpd$ and $\Mod(IC_{n+1})$ for all finite $n\geq 0$. These model structures are the semi-model structures of \cite{lanari2018semimodelstructuregrothendieckweak} but upgraded to full model structures.
\begin{definition}\label{hyp_model_struct_on_ngpd}
Let $n\geq 1$. The \emph{canonical model structure on $\ngpd$}, if it exists, has
\begin{itemize}
    \item weak equivalences given by the collection $\zW$ from Notation \ref{weak_eq_of_n_gpds},
    \item generating cofibrations  given by the set $\zI$ of Notation \ref{gen cofibrations and triv cofibrations_n_gpd}, and 
    \item and generating trivial cofibrations  given by the set $\zJ$ of Notation \ref{gen cofibrations and triv cofibrations_n_gpd}.
\end{itemize}
\end{definition}

\begin{definition}\label{hyp_model_struct_on_models}
Let $n\geq 1$. The \emph{canonical model structure on $\Mod(IC_{n+1})$}, if it exists, has
\begin{itemize}
    \item weak equivalences given by the collection $\zW'$ from Notation \ref{weak_eq_of_models},
    \item generating cofibrations  given by the set $\zI'$ of Notation \ref{gen cofibrations and triv cofibrations}, and 
    \item and generating trivial cofibrations  given by the set $\zJ'$ of Notation \ref{gen cofibrations and triv cofibrations}.
\end{itemize}
\end{definition}

The following theorem is true and we record it here for incoming results.

\begin{theorem}\label{Thm9.4}
Let $n\geq 1$. If the pushout condition holds for all models in $\Mod(IC_{n+1})$, then the canonical model structure exists.
\end{theorem}

\begin{proof}
Repeat the proof of Theorem 5.3.5 in \cite{Hen2}.
\end{proof}

The following theorem says that if one of these model structures exists for either $\ngpd$ and $\Mod(IC_{n+1})$ for some $n\geq 1$, then they both do and are Quillen equivalent.

\begin{proposition}\label{trans_model_struct_and_Quill_adj}
Let $n\geq 1$. The canonical model structure exists on $\ngpd$ if and only if it exists on $\Mod(IC_{n+1})$. If either of these equivalent conditions hold, then the following adjunction is a Quillen equivalence.
\[
\begin{tikzpicture}[>=stealth, thick, node distance=4cm]

\node (A) {$\Mod(IC_{n+1})$};
\node (B) [right of=A] {$\ngpd$};

\draw[transform canvas={yshift=0.3ex},->] (A) to node[above] {$\tr$} (B);
\draw[transform canvas={yshift=-0.3ex},<-] (A) to node[below] {$D$} (B);
\end{tikzpicture}
\]
\end{proposition}

\begin{proof}
Suppose that the canonical model structure exists on $\ngpd$. Let $X:IC_{n+1}\to\Set$ be a model, $k\geq 0$, and $f:D^k\to X$ be a map of models. Let \[
		\begin{tikzpicture}[node distance=2.5cm]
			\node(A) {$D^k$};
			\node (B) [right of=A]{$X$};
			\node(C)[below of=A]{$D^{k+1}$};
			\node (D)[right of=C]{$X_+$};
			\draw[->](B) to node[black,right=3]{$p$}(D);
			\draw[->](A) to node[above=3] {$f$}(B);
			\draw[->] (A) to node {}(C);
			\draw[->](C) to node {}(D);
		\end{tikzpicture}
		\]
    be a pushout in $\Mod(IC_{n+1})$. Then notice that this induces a pushout in $\ngpd$ as $\tr$ is a left adjoint.
    \[
    \begin{tikzpicture}[node distance=2.5cm]
			\node(A) {$D^k$};
			\node (B) [right of=A]{$\tr X$};
			\node(C)[below of=A]{$D^{k+1}$};
			\node (D)[right of=C]{$\tr X_+$};
			\draw[->](B) to node[black,right=3]{$\tr p$}(D);
			\draw[->](A) to node[above=3] {$\tr f$}(B);
			\draw[->] (A) to node {}(C);
			\draw[->](C) to node {}(D);
		\end{tikzpicture}
		\]
    By hypothesis, $\tr p$ is a weak equivalence which means $p$ is a weak equivalence by definition. Theorem \ref{Thm9.4} now says that the canonical model structure exists on $\Mod(IC_{n+1})$. 

    Suppose the canonical model structure exists on $\Mod(IC_{n+1})$. Let $X$ be an $n$-groupoid, $k\geq 0$, and $f:D^k\to X$ be a map of $n$-groupoids. Let $f':D^k\to DX$ be the following composite in $\Mod(IC_{n+1})$, where $\eta_{D^k}$ is the component of the unit of the adjunction at $D^k$.
    \[
    D^k\xrightarrow{\eta_{D^k}}DD^k\xrightarrow{Df}DX
    \]
    
    Let \[
		\begin{tikzpicture}[node distance=2.5cm]
			\node(A) {$D^k$};
			\node (B) [right of=A]{$DX$};
			\node(C)[below of=A]{$D^{k+1}$};
			\node (D)[right of=C]{$DX_+$};
			\draw[->](B) to node[black,right=3]{$p$}(D);
			\draw[->](A) to node[above=3] {$f'$}(B);
			\draw[->] (A) to node[left] {$\sigma_k$}(C);
			\draw[->](C) to node {}(D);
		\end{tikzpicture}
		\]
    be a pushout in $\Mod(IC_{n+1})$. By hypothesis, $p$ is a weak equivalence in $\Mod(IC_{n+1})$. Since $\tr$ is a left adjoint, $\tr(p)$ is a weak equivalence and the following square is a pushout in $\ngpd$.
    \[
		\begin{tikzpicture}[node distance=2.5cm]
			\node(A) {$D^k$};
			\node (B) [right of=A]{$X$};
			\node(C)[below of=A]{$D^{k+1}$};
			\node (D)[right of=C]{$\tr(DX_+)$};
			\draw[->](B) to node[black,right=3]{$\tr(p)$}(D);
			\draw[->](A) to node[above=3] {$f$}(B);
			\draw[->] (A) to node[left] {$\sigma_k$}(C);
			\draw[->](C) to node {}(D);
		\end{tikzpicture}
		\]
 Therefore the canonical model structure exists on $\ngpd$.

 Now suppose both of these equivalent conditions hold. By construction, $\tr$ sends the generating cofibrations and generating trivial cofibrations in $\Mod(IC_{n+1})$ to generating cofibrations and generating trivial cofibrations in $\ngpd$. Therefore the adjunction is a Quillen adjunction.  Notice that a map of $n$-groupoids $f:X\to Y$ is a weak equivalence if and only if $D(f)$ is a weak equivalence in $\Mod(IC_{n+1})$ since $\tr(D(f))=f$. Therefore $D$ creates weak equivalences. Let $X:IC_{n+1}\to\Set$ be a model. Then the map $\eta_X:X\to D\tr(X)$ is a weak equivalence. We conclude that the Quillen adjunction is a Quillen equivalence. 
\end{proof}

We now want to begin the process of providing an inductive strategy to prove the Homotopy Hypothesis. The fact that $0$-groupoids have a model structure corresponds to the fact that $0$-groupoids are merely sets. Our base case to begin is the $1$-groupoid case. 

\begin{lemma}\label{mod_struct_on_1_gpds}
There is an equivalence of categories where $\Gpd$ is the category of groupoids and functors between them.
\[
\mathpzc{1Gpd}\simeq \Gpd
\]
Moreover, the canonical model structure exists and is the model structure inherited from this equivalence of categories.
\end{lemma}

\begin{proof}
Use the model structure on $\Gpd$ to build a map $D:(IC_2^{\leq 1})^\op\to\Gpd$ that preserves globular sums. This induces the following adjoint equivalence.
\[
\begin{tikzpicture}[node distance=3cm]
	\node (A) {$\Gpd$};
	\node (B)[right of=A]{$\mathpzc{1Gpd}$};
	\draw[transform canvas={yshift=0.3ex},->] (A) to node[above=3] {$\Pi$} (B);
	\draw[transform canvas={yshift=-0.3ex},->,swap](B) to node[below=3] {$|-|$} (A);
\end{tikzpicture}
\]
It is now routine to show that the two model structures coincide. 
\end{proof}

We now move on from our base case to provide a necessary and sufficient case for the inductive case to be true.

\begin{proposition}\label{induc_transfer}
Let $n\geq 1$. Assuming the canonical model structure exists on $\Mod(IC_{n+1})$, then the free pushout condition is true for every model over $IC_{n+1}$ if and only if the canonical model structure exists on $\Mod(IC_{n+2})$.
\end{proposition}

\begin{proof}
Suppose the free pushout condition is true for every model over $IC_{n+1}$. Let $X:IC_{n+2}\to\Set$ be a model and $f:D^k\to UX$ be a map of $\Mod(IC_{n+1})$. We shall now consider the following pushout in $\Mod(IC_{n+1})$.
\[
		\begin{tikzpicture}[node distance=2.5cm]
			\node(A) {$D^k$};
			\node (B) [right of=A]{$UX$};
			\node(C)[below of=A]{$D^{k+1}$};
			\node (D)[right of=C]{$(UX)_+$};
			\draw[->](B) to node[black,right=3]{$p$}(D);
			\draw[->](A) to node[above=3] {$f$}(B);
			\draw[->] (A) to node[left=3] {$\sigma_k$}(C);
			\draw[->](C) to node {}(D);
		\end{tikzpicture}
		\]
By hypothesis and Theorem \ref{trans_model_struct_and_Quill_adj}, the map $p$ is a weak equivalence in $\Mod(IC_{n+1})$. The free pushout condition says that $\free(p)$ is a weak equivalence in $\Mod(IC_{n+2})$. By definition, $\tr(\free(p))$ is a weak equivalence in $\mathpzc{(n+1)Gpd}$.

Let $\epsilon_X:\free UX\to X$ be the counit component at the model $X$ for the adjunction between the categories of models and form the following pushout.
\[
		\begin{tikzpicture}[node distance=2.5cm]
			\node(A) {$\free UX$};
			\node (B) [right of=A]{$X$};
			\node(C)[below of=A]{$\free((UX)_+)$};
			\node (D)[right of=C]{$X_+$};
			\draw[->](B) to node[black,right=3]{$p$}(D);
			\draw[->](A) to node[above=3] {$\epsilon_X$}(B);
			\draw[->] (A) to node[left=3] {$\free(p)$}(C);
			\draw[->](C) to node[below=3] {$\phi$}(D);
		\end{tikzpicture}
		\]
Notice that $\tr(\epsilon_X)=1_{\tr(X)}$. Since $\tr$ is a left adjoint, the following square is a pushout in $\mathpzc{(n+1)Gpd}$.
\[
		\begin{tikzpicture}[node distance=2.5cm]
			\node(A) {$\tr(\free UX)$};
			\node (B) [right of=A]{$\tr(X)$};
			\node(C)[below of=A]{$\tr(\free((UX)_+))$};
			\node (D)[right of=C]{$\tr(X_+)$};
			\draw[->](B) to node[black,right=3]{$\tr(p)$}(D);
			\draw[->](A) to node[above=3] {$\tr(\epsilon_X)$}(B);
			\draw[->] (A) to node[left=3] {$\tr(\free(p))$}(C);
			\draw[->](C) to node[below=3] {$\tr(\phi)$}(D);
		\end{tikzpicture}
		\]

The map $\tr(\phi)$ is forced to be the identity, so that $\tr(p)$ is a weak equivalence. Therefore $p$ is a weak equivalence in $\Mod(IC_{n+2})$. We conclude that the canonical model structure exists on $\Mod(IC_{n+2})$. 

Suppose that the canonical model structure exists on $\Mod(IC_{n+2})$. Let $X:IC_{n+1}\to\Set$ be a model and $f:D^k\to X$ be a map. We shall now consider the following pushout in $\Mod(IC_{n+1})$.
\[
		\begin{tikzpicture}[node distance=2.5cm]
			\node(A) {$D^k$};
			\node (B) [right of=A]{$X$};
			\node(C)[below of=A]{$D^{k+1}$};
			\node (D)[right of=C]{$X_+$};
			\draw[->](B) to node[black,right=3]{$p$}(D);
			\draw[->](A) to node[above=3] {$f$}(B);
			\draw[->] (A) to node[left=3] {$\sigma_k$}(C);
			\draw[->](C) to node {}(D);
		\end{tikzpicture}
		\]
Since $\free$ is a left adjoint, we obtain the following induced pushout in $\Mod(IC_{n+2})$.
\[
		\begin{tikzpicture}[node distance=2.5cm]
			\node(A) {$D^k$};
			\node (B) [right of=A]{$\free X$};
			\node(C)[below of=A]{$D^{k+1}$};
			\node (D)[right of=C]{$\free(X_+)$};
			\draw[->](B) to node[black,right=3]{$\free(p)$}(D);
			\draw[->](A) to node[above=3] {$\free(f)$}(B);
			\draw[->] (A) to node[left=3] {$\sigma_k$}(C);
			\draw[->](C) to node {}(D);
		\end{tikzpicture}
		\]
The existence of the canonical model structure says that $\free(p)$ is a weak equivalence, so that the free pushout condition is satisfied for all models in $\Mod(IC_{n+1})$. 
\end{proof}

Let $n,k\geq 1$, $X:IC_{n+1}\to\Set$ be a model, and $f:D^k\to X$ be a map. Now consider the following pushout in $\Mod(IC_{n+1})$.
\[
		\begin{tikzpicture}[node distance=2.5cm]
			\node(A) {$D^k$};
			\node (B) [right of=A]{$X$};
			\node(C)[below of=A]{$D^{k+1}$};
			\node (D)[right of=C]{$X_+$};
			\draw[->](B) to node[black,right=3]{$p$}(D);
			\draw[->](A) to node[above=3] {$f$}(B);
			\draw[->] (A) to node[left=3] {$\sigma_k$}(C);
			\draw[->](C) to node {}(D);
		\end{tikzpicture}
		\]
Applying Lemma \ref{path_components_transfer}, Lemma \ref{first_hom_group_transfer}, and Lemma \ref{higher_hom_group_transfer}, we conclude that showing that $\free(p)$ is a weak equivalence boils down to showing that $\free(p)$ induces an isomorphism on the $(n+1)$ homotopy groups. In order to prove the canonical model structure exists using induction, it is sufficient to show that this is true for $n,k\geq 1$, every model $X:IC_{n+1}\to\Set$, and every map $f:D^k\to X$.

We now show that the existence of a model structure on $\mathpzc{(n+1)Gpd}$ for some $n\geq 1$ turns the appropriate adjunctions into Quillen adjunctions.

\begin{theorem}
    Let $n\geq 1$. If the canonical model structure exists on $\mathpzc{(n+1)Gpd}$, then all four of the adjunctions in the following square are Quillen adjunctions.
\[
\begin{tikzpicture}[>=stealth, thick, node distance=4cm]

\node (A) {$\Mod(IC_{n+2})$};
\node (B) [right of=A] {$\Mod(IC_{n+1})$};
\node (C) [below of=A] {$\mathpzc{(n+1)Gpd}$};
\node (D) [right of=C] {$\ngpd$};

\draw[transform canvas={yshift=0.3ex},->] (A) to node[above] {$U$} (B);
\draw[transform canvas={yshift=-0.3ex},<-] (A) to node[below] {$\free$} (B);

\draw[transform canvas={yshift=-0.3ex},->] (D) to node[below] {$D$} (C);
\draw[transform canvas={yshift=0.3ex},<-] (D) to node[above] {$\tr$} (C);

\draw[transform canvas={xshift=0.3ex},->] (A) to node[left] {$D$} (C);
\draw[transform canvas={xshift=-0.3ex},<-] (A) to node[right] {$\tr$} (C);

\draw[transform canvas={xshift=-0.3ex},->] (B) to node[left] {$\tr$} (D);
\draw[transform canvas={xshift=0.3ex},<-] (B) to node[right] {$D$} (D);

\end{tikzpicture}
\]

Moreover, the two vertical adjunctions are Quillen equivalences.
\end{theorem}

\begin{proof}
If the model structure on $\mathpzc{(n+1)Gpd}$ exists, then Proposition \ref{trans_model_struct_and_Quill_adj} combined with Proposition \ref{induc_transfer} shows that all four of the canonical model structures exist, the top horizontal adjunction is a Quillen adjunction, and the two vertical adjunctions are Quillen equivalences. Lemma \ref{path_components_transfer}, Lemma \ref{first_hom_group_transfer}, and Lemma \ref{higher_hom_group_transfer} show that $\tr$ preserves weak equivalences. Moreover, $\tr$ preserves cofibrations by construction. Therefore the bottom adjunction is a Quillen adjunction.
\end{proof}

\section{The Limiting Model Structure}

We now finish with the limiting case. Consider the chain
\[
\cdots\xrightarrow{\tr}\mathpzc{(n+1)Gpd}\xrightarrow{\tr}\ngpd\xrightarrow{\tr}\cdots\xrightarrow{\tr}\mathpzc{1\Gpd}\xrightarrow{\tr}\mathpzc{0Gpd}
\]
of categories. There is a map 
\[
\tr_n:\inftygpd\to\ngpd
\]
for all $n\geq 1$, which is a left adjoint, induced by the chain and limit structure on the underlying theories.
\begin{definition}
We say a map $f:X\to Y$ of $\infty$-groupoids is a weak equivalence if $\tr_n(f)$ is a weak equivalence of $n$-groupoids for all $n\geq 0$.
\end{definition}

\begin{notation}\label{weak_eq_of_inf_gpds}
We write $\zW'$ to denote the collection of all weak equivalences of $\infty$-groupoids.
\end{notation}

\begin{lemma}
The collection of weak equivalences for $\infty$-groupoids presented here, i.e $\zW'$, coincides with the collection of weak equivalences of Definition 4.17 of \cite{Dim1}.
\end{lemma}

\begin{proof}
Let $f:X\to Y$ be an element of $\zW'$. Suppose $n\geq 1$. Then $\tr_n(f)$ is a weak equivalence of $\infty$-groupoids. By definition, $f$ induces a bijection on path components and homotopy groups up to dimension $n$. Since this holds for all $n\geq 1$, $f$ is a weak equivalence in the sense of Definition 4.17 of \cite{Dim1}.

Suppose $f$ is a weak equivalence in the sense of Definition 4.17 of \cite{Dim1}. Then $\tr_n(f)$ is a weak equivalence merely by definition. Therefore $f\in\zW'$. We conclude that the two collections coincide.
\end{proof}

\begin{definition}
The canonical model structure on $\inftygpd$, if it exists, has
\begin{itemize}
     \item weak equivalences given by the collection $\zW'$ from Notation \ref{weak_eq_of_inf_gpds},
    \item generating cofibrations  given by the set $\zI'$ of Notation \ref{gen cofibrations and triv cofibrations_n_gpd}, and 
    \item and generating trivial cofibrations  given by the set $\zJ'$ of Notation \ref{gen cofibrations and triv cofibrations_n_gpd}.
\end{itemize}
\end{definition}

\begin{theorem}\label{fin_iff_inf}
The  canonical model structure on $\inftygpd$ exists if and only if the canonical model structure exists on $\ngpd$ for all finite $n$.
\end{theorem}

\begin{proof}
    Suppose the canonical model structure exists on $\ngpd$ for all finite $n$. Let $X$ be an $\infty$-groupoid and $k\geq 0$.  Form the following pushout in $\inftygpd$. 
    \[
		\begin{tikzpicture}[node distance=2.5cm]
			\node(A) {$D^k$};
			\node (B) [right of=A]{$X$};
			\node(C)[below of=A]{$D^{k+1}$};
			\node (D)[right of=C]{$X_+$};
			\draw[->](B) to node[black,right=3]{$p$}(D);
			\draw[->](A) to node[above=3] {$f$}(B);
			\draw[->] (A) to node[left=3] {$\sigma_k$}(C);
			\draw[->](C) to node {}(D);
		\end{tikzpicture}
		\]
Notice that as $\tr_n$ is a left adjoint for all $n\geq 0$, the following square is a pushout in $\ngpd$.
  \[
		\begin{tikzpicture}[node distance=2.5cm]
			\node(A) {$D^k$};
			\node (B) [right of=A]{$\tr_n(X)$};
			\node(C)[below of=A]{$D^{k+1}$};
			\node (D)[right of=C]{$\tr_n(X_+)$};
			\draw[->](B) to node[black,right=3]{$\tr_n(p)$}(D);
			\draw[->](A) to node[above=3] {$\tr_n(f)$}(B);
			\draw[->] (A) to node[left=3] {$\sigma_k$}(C);
			\draw[->](C) to node {}(D);
		\end{tikzpicture}
		\]
Since the canonical model structure exists on $\ngpd$, $\tr_n(p)$ is a weak equivalence in $\ngpd$. By definition, $p$ is a weak equivalence in $\inftygpd$.

Suppose that the canonical model structure exists on $\inftygpd$. Let $n$ be a non-negative integer. Let $X$ be an $n$-groupoid and $k\geq 0$. Consider the following pushout in $\inftygpd$, where $D_n$ is the right adjoint to $\tr_n$.
    \[
		\begin{tikzpicture}[node distance=2.5cm]
			\node(A) {$D^k$};
			\node (B) [right of=A]{$D_n(X)$};
			\node(C)[below of=A]{$D^{k+1}$};
			\node (D)[right of=C]{$X_+$};
			\draw[->](B) to node[black,right=3]{$p$}(D);
			\draw[->](A) to node[above=3] {$f$}(B);
			\draw[->] (A) to node[left=3] {$\sigma_k$}(C);
			\draw[->](C) to node {}(D);
		\end{tikzpicture}
		\]
By hypothesis, $p$ is a weak equivalence. There is an induced pushout in $\ngpd$.
    \[
		\begin{tikzpicture}[node distance=2.5cm]
			\node(A) {$D^k$};
			\node (B) [right of=A]{$X$};
			\node(C)[below of=A]{$D^{k+1}$};
			\node (D)[right of=C]{$\tr_n(X_+)$};
			\draw[->](B) to node[black,right=3]{$\tr_n(p)$}(D);
			\draw[->](A) to node[above=3] {$\tr_n(f)$}(B);
			\draw[->] (A) to node[left=3] {$\sigma_k$}(C);
			\draw[->](C) to node {}(D);
		\end{tikzpicture}
		\]
Therefore $\tr_n(p)$ is a weak equivalence in $\ngpd$. We conclude that the canonical model structure exists on $\ngpd$ by using properties of adjunctions.
\end{proof}

\begin{corollary}
Assuming the free pushout condition holds every model over $IC_{n+1}$ for all $n\geq 1$,the following adjunction is a Quillen equivalence of model categories.
\[
\begin{tikzpicture}[node distance=3cm]
	\node (A) {$\Top$};
	\node (B)[right of=A]{$\inftygpd$};
	\draw[transform canvas={yshift=0.3ex},->] (A) to node[above=3] {$\Pi$} (B);
	\draw[transform canvas={yshift=-0.3ex},->,swap](B) to node[below=3] {$|-|$} (A);
\end{tikzpicture}
\]
Moreover, there is an induced equivalence of categories between the homotopy category of $n$-groupoids and the homotopy category of homotopy $n$-types for all $n\geq 0$, including $n=\infty$. 
\end{corollary}

\begin{proof}
Apply Lemma \ref{mod_struct_on_1_gpds}, Proposition \ref{induc_transfer}, Theorem \ref{fin_iff_inf}, Corollary 5.3.13 of \cite{Hen2}, and Proposition 4.5 of \cite{HenLan}.
\end{proof}

This is to say that the Generalized Homotopy Hypothesis is true.

\bibliography{math}

@misc{Malt,
      title={Grothendieck $\infty$-groupoids, and still another definition of $\infty$-categories}, 
      author={Georges Maltsiniotis},
      year={2010},
      eprint={1009.2331},
      archivePrefix={arXiv},
      primaryClass={math.CT},
  howpublished=" Accessed from \url{https://arxiv.org/abs/1009.2331}", 
}

@InProceedings{Beck_Dist,
author="Beck, Jon",
editor="Eckmann, B.",
title="Distributive laws",
booktitle="Seminar on Triples and Categorical Homology Theory",
year="1969",
publisher="Springer Berlin Heidelberg",
address="Berlin, Heidelberg",
pages="119--140",
isbn="978-3-540-36091-9"
}

@article {Dim1,
    AUTHOR = {Ara, Dimitri},
     TITLE = {On the homotopy theory of {G}rothendieck {$\infty$}-groupoids},
   JOURNAL = {J. Pure Appl. Algebra},
  FJOURNAL = {Journal of Pure and Applied Algebra},
    VOLUME = {217},
      YEAR = {2013},
    NUMBER = {7},
     PAGES = {1237--1278},
      ISSN = {0022-4049,1873-1376},
   MRCLASS = {18G55 (18D05 55P10)},
  MRNUMBER = {3019735},
MRREVIEWER = {Josep\ Elgueta},
       DOI = {10.1016/j.jpaa.2012.10.010},
        howpublished=" Accessed from \url{https://doi.org/10.1016/j.jpaa.2012.10.010}",
}

@article {Garn2,
    AUTHOR = {Garner, Richard},
     TITLE = {A homotopy-theoretic universal property of {L}einster's operad
              for weak {$\omega$}-categories},
   JOURNAL = {Math. Proc. Cambridge Philos. Soc.},
  FJOURNAL = {Mathematical Proceedings of the Cambridge Philosophical
              Society},
    VOLUME = {147},
      YEAR = {2009},
    NUMBER = {3},
     PAGES = {615--628},
      ISSN = {0305-0041,1469-8064},
   MRCLASS = {18G55 (18D50 55U35)},
  MRNUMBER = {2557146},
MRREVIEWER = {Philippe\ Gaucher},
       DOI = {10.1017/S030500410900259X},
        howpublished=" Accessed from \url{https://doi.org/10.1017/S030500410900259X}",
}

@book {Lein,
    AUTHOR = {Leinster, Tom},
     TITLE = {Higher operads, higher categories},
    SERIES = {London Mathematical Society Lecture Note Series},
    VOLUME = {298},
 PUBLISHER = {Cambridge University Press, Cambridge},
      YEAR = {2004},
     PAGES = {xiv+433},
      ISBN = {0-521-53215-9},
   MRCLASS = {18D50 (55P48)},
  MRNUMBER = {2094071},
MRREVIEWER = {R.\ H.\ Street},
       DOI = {10.1017/CBO9780511525896},
        howpublished=" Accessed from \url{https://doi.org/10.1017/CBO9780511525896}",
}

@article {EC1,
    AUTHOR = {Cheng, Eugenia},
     TITLE = {Iterated distributive laws},
   JOURNAL = {Math. Proc. Cambridge Philos. Soc.},
  FJOURNAL = {Mathematical Proceedings of the Cambridge Philosophical
              Society},
    VOLUME = {150},
      YEAR = {2011},
    NUMBER = {3},
     PAGES = {459--487},
      ISSN = {0305-0041,1469-8064},
   MRCLASS = {18C20 (18D05)},
  MRNUMBER = {2784770},
MRREVIEWER = {Enrico\ Vitale},
       DOI = {10.1017/S0305004110000599},
        howpublished=" Accessed from \url{https://doi.org/10.1017/S0305004110000599}",
}

@article {Geo2,
    AUTHOR = {Maltsiniotis, Georges},
     TITLE = {La th\'eorie de l'homotopie de {G}rothendieck},
   JOURNAL = {Ast\'erisque},
  FJOURNAL = {Ast\'erisque},
    VOLUME = {301},
      YEAR = {2005},
     PAGES = {vi+140},
      ISSN = {0303-1179,2492-5926},
   MRCLASS = {18G55 (14F35 18G10 55U10 55U35)},
  MRNUMBER = {2200690},
MRREVIEWER = {J\v ir\'i\ Rosick\'y},
}

@misc {BarrWells,
    AUTHOR = {Barr, Michael and Wells, Charles},
     TITLE = {Toposes, triples and theories},
      NOTE = {Corrected reprint of the 1985 original [MR0771116]},
   JOURNAL = {Repr. Theory Appl. Categ.},
  FJOURNAL = {Reprints in Theory and Applications of Categories},
    NUMBER = {12},
      YEAR = {2005},
     PAGES = {x+288},
   MRCLASS = {18-02 (03G30 18B25 18C10 18C15)},
  MRNUMBER = {2178101},
}

@article {Bo2,
    AUTHOR = {Bourke, John},
     TITLE = {Note on the construction of globular weak omega-groupoids from
              types, topological spaces{$\ldots$}},
   JOURNAL = {Cah. Topol. G\'eom. Diff\'er. Cat\'eg.},
  FJOURNAL = {Cahiers de Topologie et G\'eom\'etrie Diff\'erentielle
              Cat\'egoriques},
    VOLUME = {57},
      YEAR = {2016},
    NUMBER = {4},
     PAGES = {281--294},
      ISSN = {1245-530X,2681-2363},
   MRCLASS = {18D05},
  MRNUMBER = {3616553},
MRREVIEWER = {Philippe\ Gaucher},
}

@article {HenLan,
    AUTHOR = {Henry, Simon and Lanari, Edoardo},
     TITLE = {On the homotopy hypothesis for 3-groupoids},
   JOURNAL = {Theory Appl. Categ.},
  FJOURNAL = {Theory and Applications of Categories},
    VOLUME = {39},
      YEAR = {2023},
     PAGES = {Paper No. 26, 735--768},
      ISSN = {1201-561X},
   MRCLASS = {18N20 (18M90 18N40 55U35)},
}

@misc{lanari2018semimodelstructuregrothendieckweak,
      title={A semi-model structure for Grothendieck weak 3-groupoids}, 
      author={Edoardo Lanari},
      year={2018},
      eprint={1809.07923},
      archivePrefix={arXiv},
      primaryClass={math.CT},
        howpublished=" Accessed from \url{https://arxiv.org/abs/1809.07923}", 
}

@misc{Hen2,
      title={Algebraic models of homotopy types and the homotopy hypothesis}, 
      author={Simon Henry},
      year={2016},
      eprint={1609.04622},
      archivePrefix={arXiv},
      primaryClass={math.CT},
  howpublished=" Accessed from \url{https://arxiv.org/abs/1609.04622}", 
}

@article{lane2,
title = {Towards a globular path object for weak {$\infty$}-groupoids},
journal = {Journal of Pure and Applied Algebra},
volume = {224},
number = {2},
pages = {630-702},
year = {2020},
issn = {0022-4049},
doi = {https://doi.org/10.1016/j.jpaa.2019.06.004},
 howpublished=" Accessed from \url{https://www.sciencedirect.com/science/article/pii/S0022404919301501}",
author = {Edoardo Lanari}
}
\bibliographystyle{amsplain}
\end{document}